\documentclass[11pt,a4paper,reqno]{amsart}
\usepackage{amsmath,amssymb,amsfonts,epsfig,mathrsfs,cite, hyperref}
\usepackage[T1]{fontenc}
\usepackage{color}
\usepackage{array}
\usepackage{amsthm}
\usepackage{amstext}
\usepackage{graphicx}
\usepackage{setspace}

\usepackage[title]{appendix}

\usepackage{booktabs}

\usepackage{tcolorbox}

\usepackage{float}

\makeatletter
\@namedef{subjclassname@2020}{%
  \textup{2020} Mathematics Subject Classification}
\makeatother

\usepackage[margin=2.5cm]{geometry}
\usepackage{color}
\usepackage{enumitem}
\usepackage{amscd,psfrag}
\usepackage{yhmath}
\usepackage[mathscr]{eucal}
\usepackage{comment}

\allowdisplaybreaks[4]

\usepackage{slashed}

\makeatletter
\pdfpageheight\paperheight
\pdfpagewidth\paperwidth

\usepackage{epstopdf}
\usepackage{indentfirst}	

\usepackage[normalem]{ulem}
\theoremstyle{plain}
\newtheorem{definition}{Definition}
\newtheorem{theorem}[definition]{Theorem}
\newtheorem*{theorem*}{Theorem}

\newtheorem*{remark*}{Remark}
\newtheorem*{sideremark*}{Side Remark}

\newtheorem*{claim*}{Claim}
\newtheorem*{lemma*}{Lemma}
\newtheorem*{q*}{Question}
\newtheorem{lemma}[definition]{Lemma}
\newtheorem{corollary}[definition]{Corollary}
\newtheorem*{corollary*}{Corollary}
\newtheorem{example}[definition]{Example}
\newtheorem{proposition}[definition]{Proposition}

\newtheorem{conjecture}[definition]{Conjecture}

\newcommand{\R}{\mathbb{R}}

\newcommand{\na}{\nabla}
\newcommand{\emb}{\hookrightarrow}

\newcommand{\p}{\partial}
\newcommand{\loc}{{\rm loc}}
\newcommand{\weak}{\rightharpoonup}
\newcommand{\e}{\epsilon}

\newcommand{\dd}{{\rm d}}

\newcommand{\G}{\Gamma}

\newcommand{\M}{{\mathcal{M}}}

\newcommand{\mres}{\mathbin{\vrule height 1.6ex depth 0pt width
0.13ex\vrule height 0.13ex depth 0pt width 1.3ex}}

\newcommand{\K}{{\mathcal{K}}}

\newcommand{\bigs}{\mathscr{S}}

\def\Xint#1{\mathchoice
{\XXint\displaystyle\textstyle{#1}}%
{\XXint\textstyle\scriptstyle{#1}}%
{\XXint\scriptstyle\scriptscriptstyle{#1}}%
{\XXint\scriptscriptstyle\scriptscriptstyle{#1}}%
\!\int}
\def\XXint#1#2#3{{\setbox0=\hbox{$#1{#2#3}{\int}$ }
\vcenter{\hbox{$#2#3$ }}\kern-.6\wd0}}

\def\dashint{\Xint-}

\newcommand{\N}{{\mathcal{N}}}

\newcommand{\tdn}{{\mathcal{T}}_\delta(\N)}

\newcommand{\ball}[2]{\mathbf{B}_{#1}({#2})}

\newcommand{\spt}[1]{{\bf spt}({#1})}

\newcommand{\E}{\mathcal{E}}

\newcommand{\bigo}{\mathcal{O}}

\newcommand{\stwo}{{\bf S}^2}
\newcommand{\bthree}{{\bf B}^3}
\newcommand{\sing}{{\rm sing}}
\newcommand{\singu}{\sing(u)}

\newcommand{\leb}{{\mathcal{L}}}
\newcommand{\hau}{{\mathcal{H}}}

\newcommand{\bone}{{\mathbf{B}_1}}

\newcommand{\F}{{\mathcal{F}}}

\title{A brief introduction to the regularity theory of minimising harmonic maps} 

\author{Siran Li}

\address{Siran Li: School of Mathematical Sciences $\&$ CMA-Shanghai, Shanghai Jiao Tong University, No.~6 Science Buildings,
800 Dongchuan Road, Minhang District, Shanghai, China (200240)}

\email{\texttt{siran.li@sjtu.edu.cn}} 


\date{\today}

\begin{document}

\begin{abstract}
We introduce various classical, foundational results on the regularity theory of harmonic maps, with focuses on the energy minimising harmonic maps. Topics include inner and outer variations, the monotonicity identity, the $\varepsilon$-regularity theorem, blow-ups and tangent maps, and the dimension estimate and stratification of singular sets. We also discuss the analogous theory for stationary harmonic maps and mention some significant recent developments.

\end{abstract}
\maketitle

%\tableofcontents

\section{Description}

These notes have been written based on a three-hour mini-course delivered by the author at the \emph{Winter School on Geometric Measure Theory
Rectifiability \textit{vs.} Pure Unrectifiability}, held at Westlake University, Hangzhou, from 1st to 6th February, 2026. The mini-course has been designed to provide a brief and accessible introduction to foundations of the regularity theory of harmonic maps developed about 1980-90s. The analytic theory of harmonic maps has long been a central topic in geometric variational problems, nonlinear elliptic systems, and geometric measure theory, and has also found numerous applications in physics, material sciences, and biology, etc. Our presentation focuses mainly on the Dirichlet energy minimising harmonic maps; nevertheless, whenever adequate, we shall hinge on the more general --- and more challenging --- theory of stationary harmonic maps. Efforts are made to provide a balanced combination of both the PDE and geometric variational aspects of the harmonic map theory.

These notes may serve as a streamlined version for (though certainly not a substitute for) most of the materials in the first three chapters of L. Simon's 1996 notes~\cite{simon}. We slightly reorganise and restructure the materials therein, and present alternative or detailed arguments for several results, especially Proposition~\ref{propn: outer variation}, Theorem~\ref{thm: codimenion 3}, and Lemma~\ref{lem: tech}. However, we do not claim originality for any material in our notes. Meanwhile, we do not venture to touch on every aspect of harmonic map theory. The topics and references in these notes are not exhaustive: we apologise for any omission. We bring to the reader's attention that many problems listed in R. Hardt's 1997 report~\cite{hardt} has remained open until today, including those pertaining to low-dimensional harmonic maps, \textit{e.g.}, those from $\bthree$ into $\stwo$. We also recommend that the reader study the 2008 monograph by Lin--Wang~\cite{lw}, which provides a comprehensive survey of the theory of harmonic maps as it stood prior to the developments by Naber--Valtorta~\cite{nv}.

In \S\ref{sec: hm} we study the variational and PDE formulations for harmonic maps, highlighting the distinction between minimising, stationary, and weakly harmonic maps. In \S\ref{sec: mono} we present the monotonicity formulae and the $\e$-regularity theory for stationary harmonic maps. (Note, though, that the proof in \S\ref{sec: mono} of $\e$-regularity is given only for minimising harmonic maps, as in Simon's notes~\cite{simon}.) Then, in \S\ref{sec: blowup}, we introduce the blow-up analysis of minimising harmonic maps at singular points. In particular, by analysing the tangent maps obtained as blow-up limits, we prove a dimension estimate for the singular set and introduce the stratification of singular set based on the degrees of symmetry of the tangent maps. Finally, we briefly discuss in \S\ref{sec: final} some more recent developments on the analytic theory of (mostly, stationary) harmonic maps.

\section{Minimising, stationary, and weakly harmonic maps}\label{sec: hm}

As the name suggests, a minimising harmonic map from $\Omega \subset \R^n$ to a manifold $\N$ is a minimiser of the Dirichlet energy. In these notes, we restrict ourselves to maps $u: \Omega \subset \R^n \to (\N,g)$ from a Euclidean domain to a Riemannian manifold. The regularity theory for harmonic maps in the general case $u: (\M,h)\to(\N,g)$, where $(\M,g)$ is another Riemannian manifold, follows from direct adaptations of the case of Euclidean domains.\footnote{Somewhat a rule of thumb: in the study of maps $u:(\M,h) \to (\N,g)$ (for harmonic maps in our case, or for the analysis of Sobolev spaces $W^{1,p}(\M;\N)$...), it is always the target manifold $\N$ that causes more complications, rather than the source manifold $\M$.} See, \emph{e.g.},\cite[Section~5]{lin}.

Let $u: \Omega \to \N$ and $\Omega' \subset \Omega$, and consider 
\begin{equation}\label{Dir energy}
    \E_{\Omega'}[u]:= \frac{1}{2}\int_{\Omega'} |\na u|^2\,\dd x,
\end{equation}
with the norm evaluated with respect to the Euclidean metric on $\Omega$ and the Riemannian metric on $\N$. We take $u$ in the Sobolev space $W^{1,2}(\Omega;\N)$ defined as follows: for a $C^k$-manifold $\N$ with $k \geq 3$, by Nash's embedding theorem, $\N$ can be regarded as a submanifold of the Euclidean space $\R^p$ for some $p$. We set
\begin{align*}
W^{1,2}(\Omega;\N) := \Big\{u \in W^{1,2}(\Omega;\R^p):\, \text{$u(x) \in \N$ for a.e. $x \in \Omega$}\Big\}.
\end{align*}
We say that $u \in W^{1,2}(\Omega;\N)$ is a minimising harmonic map on $\Omega$ iff $\E_{\Omega}[u] \leq \E_{\Omega}[u']$ for any $u'\in W^{1,2}(\Omega \subset \R^n; \N)$ with $u=u'$ on $\p{\Omega}$ in the sense of trace. By considering suitable cutoff functions, this is equivalent to $\E_{\ball{\rho}{y}}[u] \leq \E_{\ball{\rho}{y}}[u']$ for any $u'\in W^{1,2}(\Omega \subset \R^n; \N)$ with $u=u'$ on $\p{\ball{\rho}{y}}$ in the sense of trace.

As with many variational problems, an important approach to the study of analytic or geometric structures of minimisers for certain energy functionals is to analyse the corresponding \emph{Euler--Lagrange equations}, namely the PDE obtained via
\begin{align*}
    \frac{d}{ds}\bigg|_{s=0} \E_{\Omega}\left[u_s\right] = 0\qquad \text{for all variations $\{u_s\}_{s \in ]-\e,\e[}$ of $u$}.
\end{align*}
Here, by ``variations'' we mean $\{u_s\}_{s \in ]-\e,\e[} \subset W^{1,2}(\Omega;\R^p)$ for suitably small $\e>0$ such that $u_0\equiv u$ and $s \mapsto  u_s$ is a $C^1$-map from $]-\e,\e[$ to $W^{1,2}(\Omega;\R^p)$.

We consider two types of variations, one obtained by perturbing the range variable (``outer variations''),  and the other by perturbing the domain variable (``inner variations'').

\subsection{Outer variations}
Let $\N$ be a $C^k$-submanifold of $\R^p$; $k \geq 3$ and $\dim\N=N\leq p$. By an application of the inverse function theorem, the nearest point projection $\Pi$ onto $\N$ is a well defined $C^{k-1}$-map over the $\delta$-tubular neighbourhood of $\N$ for some $\delta = \delta(\N)>0$:
\begin{align*}
    &\Pi: \tdn\equiv\left\{x \in \R^p:\, {\rm dist}(x,\N)<\delta\right\} \longrightarrow \N,\\
    & \Pi(x) \text{ is the unique point on $\N$ such that $|x-\Pi(x)|={\rm dist}(x,\N)$}.
\end{align*}

An outer variation is given by 
\begin{equation}
    u^{\rm outer}_s (x) := \Pi\Big(u(x)+s\zeta(x)\Big) : \Omega \to \N
\end{equation}
for some test function $\zeta$. If $u$ is an energy minimiser over any balls in $
\Omega$, then \begin{equation}\label{outer var}
    \frac{d}{ds}\bigg|_{s=0} \E_{\ball{\rho}{y}}\left[u^{\rm outer}_s\right] = 0\qquad\text{for any } \ball{\rho}{y} \Subset \Omega \text{ and } \zeta \in C^\infty_0\left(\ball{\rho}{y}; \R^p\right).
\end{equation}
Here $|s|$ is so small that $u+s\zeta$ takes values in $\tdn$. 

\begin{proposition}\label{propn: outer variation}\footnote{It is customary is PDE literature to refer to Equations~\eqref{weak harmonic map eq, 0}, \eqref{weak harmonic map eq, 1}, or \eqref{weak harmonic map eq, 2} as ``the harmonic map equation''.}
Suppose $u \in W^{1,2}(\Omega;\N)$ is a critical point of the Dirichlet energy under outer variations; \emph{i.e.}, the condition~\eqref{outer var} holds. Then 
    \begin{align}\label{weak harmonic map eq, 0}
(\Delta u)^T = 0\qquad\text{ on }\Omega,
    \end{align}
where $(\Delta u)^T$ is the orthogonal projection of $\Delta u:\Omega \to \R^p$ to $T\N$. This is equivalent to \begin{equation}\label{weak harmonic map eq, 1}
-\p_i\p_i u +  {\rm Hess}_u\Pi(\p_iu,\p_iu) = 0 \qquad\text{ on }\Omega
\end{equation}
and
\begin{equation}\label{weak harmonic map eq, 2}
\p_i\p_i u +  A_u(\p_iu,\p_iu) = 0 \qquad\text{ on }\Omega,
\end{equation}
where $A_u: \G(T_u\N)\times\G(T_u\N)\to\G(T_u\N^\perp)$ is the second fundamental form of $\N \subset \R^p$. 
\end{proposition}

We first collect some rudiments of Euclidean geometry/multivariable calculus. Throughout,   denote by $D$ the usual derivative (\emph{i.e.}, the Levi-Civita connection) on the ambient Euclidean space $\R^p$. For two vectorfields $v, w \in \G(T\R^p)$, it holds that $D_{v}w=v^\alpha \p_\alpha w = v \cdot Dw$. Then, by writing $\Pi = (\Pi^1,\cdots, \Pi^p)^\top=\Pi^\beta\p_\beta$, we obtain for each $v \in \G(T\R^p)$ and $y \in \tdn$ that 
\begin{align}\label{T-proj}
    v\cdot D\Pi\big|_y &= v^\alpha \p_\alpha \Pi^\beta\p_\beta\big|_y=D_v\Pi\big|_y=d_y\Pi(v) \nonumber\\
    & = v^T\big|_y := \text{the orthogonal projection of $v$ on $T_{\Pi(y)}\N$},
\end{align}
where the exterior differential is understood as $d:T_y\R^p \to T_{\Pi(y)}\N$. Throughout, we shall label $i,j,k,\ldots \in \{1,\ldots,d\}$ and $\alpha,\beta,\gamma,\ldots \in \{1,\ldots, p\}$. 
Here and hereafter, we denote the orthogonal decomposition of a vectorfield
\begin{align*}
    v = v^T + v^\perp,\qquad v^T \in \G(T\N) \text{ and } v^\perp \in \G(T\N^\perp).
\end{align*}

The Hessian of $\Pi:\tdn\subset\R^p\to\N$,  is the bilinear form  
\begin{align*}
    {\rm Hess}_y\, \Pi(v,w) := \left(D_vD_w\Pi\right)\big|_y -\big(D_{D_vw}\Pi\big)\big|_y
\end{align*}
for each $y \in \tdn$ and $v,w \in \G(T\R^p)$. The presence of the second term ensures the tensoriality of Hessian. It follows that
\begin{align}\label{hess-a}
{\rm Hess}_y\, \Pi(v,w) = D_v (w^T)\big|_y - \big(D_vw\big)^T\Big|_y\qquad \text{for $y \in \tdn$; $v,w \in \G(T\R^p)$}. 
\end{align}
The following two special cases are of particular relevance for us:
\begin{itemize}
    \item 
If $v,w \in \G(T\N)$, then   \begin{align}\label{hess-b}
{\rm Hess}_y\, \Pi(v,w) = D_v w\big|_y - \big(D_vw\big)^T\Big|_y = \big(D_vw\big)^\perp\Big|_y. 
\end{align}
\item 
If, furthermore, $v=\p_j u$ and $w = \p_k u$, then
\begin{align}\label{hess-c}
{\rm Hess}_u \Pi(\p_j u, \p_ku) = \left(\p_j\p_ku\right)^\perp.
\end{align}
Indeed, differentiation of $\Pi\circ u = u$ gives us $\p_j u^\alpha \p_\alpha \Pi\big|_u =\p_ju$, and a further differentiation yields that $   (\p_j\p_k u^\alpha)\left(\p_\alpha\Pi\circ u\right) + \left(\p_ju^\alpha\right)(\p_k u^\beta)\left(\p_{\alpha\beta}\Pi \circ u\right) = \p_j\p_ku.$ We then conclude \eqref{hess-c} by noting that $(\p_j\p_k u^\alpha)\left(\p_\alpha\Pi\circ u\right) \equiv D_{\p_j\p_ku}\Pi\big|_u \equiv ({\p_j\p_ku})^T$.
\end{itemize}

\begin{proof}[Proof of Proposition~\ref{propn: outer variation}]
Let us compute by the chain rule, Taylor's expansion,  the definition of Hessian, and \eqref{T-proj}:
\begin{align*}
    \p_iu^{\rm outer}_s &= \p_i(u+s\zeta) \cdot D\Pi \big|_{u+s\zeta}\\
    &= \p_i(u+s\zeta)\cdot \Big\{D\Pi\big|_{u} + sD^2\Pi\big|_u \cdot \zeta + \bigo(s^2) \Big\}\\
    &= \p_iu + s\bigg\{\left(\p_i\zeta\right)^T + {\rm Hess}_u  \Pi(\p_i u,\zeta) \bigg\}+\bigo(s^2).
\end{align*}
Substituting this into \eqref{outer var} and noting that $\p_iu=(\p_iu)^T$, we have
\begin{align}\label{test}
0 = \int_\Omega \p_i u\cdot\Big\{\p_i\zeta+{\rm Hess}_u  \Pi(\p_i u,\zeta) \Big\} \,\dd x\quad \text{for any } \zeta \in C^\infty_0(\ball{\rho}{y};\R^p).
\end{align}

Notice that
\begin{align*}
    \p_iu \cdot {\rm Hess}_u  \Pi(\p_i u,\zeta) &= \p_iu \cdot {\rm Hess}_u  \Pi\left(\p_i u,\zeta^T+\zeta^\perp\right) = \p_iu \cdot {\rm Hess}_u  \Pi\left(\p_iu,\zeta^\perp\right)\\
    &= -\zeta^\perp \cdot {\rm Hess}_u\Pi(\p_iu,\p_iu)= -\zeta \cdot {\rm Hess}_u\Pi(\p_iu,\p_iu),
\end{align*}
thanks to \eqref{hess-b} and the following observation: for $a,b\in\G(T\N)$ and $\eta \in \G(T\N^\perp)$, taking $D_b$ to $a\cdot\eta=0$ leads to $0=a\cdot D_b\eta + \eta \cdot D_ba$. Thus, by integration by parts, we deduce from \eqref{test} that
\begin{equation*}
-\p_i\p_i u +  {\rm Hess}_u\Pi(\p_iu,\p_iu) = 0,
\end{equation*}
which is \eqref{weak harmonic map eq, 1}. Meanwhile, by \eqref{hess-c} we have ${\rm Hess}_u\Pi(\p_iu,\p_iu) = (\Delta u)^\perp$, so \eqref{weak harmonic map eq, 0} follows. Finally, by the definition of the second fundamental form, we have $
    A_u(\p_iu,\p_iu) = -\left(D_{\p_iu}(\p_iu)\right)^\perp$.  But $\p_iu \equiv (\p_iu)^T$, so by the definition of Hessian in  \eqref{hess-a}, 
\begin{align*}
    A_u(\p_iu,\p_iu) = -{\rm Hess}_u(\p_iu, \p_iu). 
\end{align*}
This proves \eqref{weak harmonic map eq, 2}.  \end{proof}

\subsection{Inner variations} 
Now we consider the \emph{inner variations}:
\begin{equation}\label{inner var, def}
    u^{\rm in}_s(x) := u\left(x+s\varphi(x)\right),\qquad \varphi = \left(\varphi^1, \cdots, \varphi^n\right)^\top \in C^\infty_0\left(\Omega;\R^n\right).
\end{equation}
 Since $u$ is an energy minimiser, it holds that
\begin{equation}\label{inner var}
    \frac{d}{ds}\bigg|_{s=0} \E_{\Omega}\left[u^{\rm in}_s\right] = 0\qquad\text{for any } \varphi \in C^\infty_0\left(\Omega;\R^n\right).
\end{equation}

By the chain rule, we compute that
\begin{align*}
    \p_i \left(u^{\rm in}_s(x)\right) &= (\p_iu)(x+s\varphi(x)) + s \sum_{j=1}^n \p_i\varphi^j(x) \p_ju(x+s\varphi(x)).
\end{align*}
Heuristic arguments involving the Taylor expansion for $u$ up to the second order show that right-hand side equals
\begin{align}\label{heuristic}
(\p_iu)(x) + \left\{s\sum_{k=1}^n\varphi^k(x) \p_k\p_iu(x) + \bigo(s^2) \right\}+ \left\{s\sum_{j=1}^n \p_i\varphi^j(x) \p_ju(x) + \bigo(s^2)\right\}.
\end{align}
Then, taking any $\ball{\rho}{y} \Subset \Omega$ and specialising to $\varphi \in C^\infty_0(\ball{\rho}{y};\R^n)$, one deduces that
\begin{align}\label{stationary hm}
\frac{1}{2} \frac{d}{ds}\bigg|_{s=0} \E_{\ball{\rho}{y}}\left[u^{\rm in}_s(x)\right] &= \int_{\ball{\rho}{y}} \p_iu(x)\cdot \sum_{k=1}^n  \left(\varphi^k(x) \p_k\p_iu(x)+\p_i\varphi^k\p_k u\right)\,\dd x\nonumber\\
&= \int_{\ball{\rho}{y}} \left\{\frac{1}{2}\varphi\cdot\na \left|\na u\right|^2 +\na\varphi:(\na u \otimes \na u)\right\}\,\dd x \nonumber\\
&= \int_{\ball{\rho}{y}} \na\varphi:\left(\na u \otimes \na u - \frac{1}{2}|\na u|^2\,{\bf Id}\right)\,\dd x = 0.
\end{align}
Equivalently, in local coordinates, 
\begin{align}\label{stationary hm, coord}
\sum_{i,j,k=1}^n     \int_{\ball{\rho}{y}} \p_i \varphi^j \left\{\p_ju^k \p_iu^k - \frac{1}{2}|\na u|^2\delta_{ij}\right\}\,\dd x = 0
\end{align}
where $|\na u|^2 = \sum_{i,j=1}^n\left(\p_i u^j\right)^2$. Another way to write~\eqref{stationary hm} is 
\begin{align}
    {\rm Div}\left(\na u \otimes \na u - \frac{1}{2}|\na u|^2\,{\bf Id}\right)=0\qquad \text{on } \ball{\rho}{y},
\end{align}
where ${\rm Div}$ is the row-wise divergence of a matrix field.

To make the above heuristic arguments rigorous for $u \in W^{1,2}$, observe that the mapping $\Psi_s(x):=x+s\varphi(x)$ is a $C^\infty$-diffeomorphism of $\ball{\rho}{y}$ into itself for $|s|$ sufficiently small and $\varphi \in C^\infty_0(\ball{\rho}{y};\R^n)$ given. Then we have
\begin{align*}
\det(D\Psi_s) = \det(I + sD\varphi) = 1+s\,{\rm Tr}(D\varphi) + \bigo(s^2) = 1+s\,{\rm div}(\varphi) + \bigo(s^2)
\end{align*}
and $\Phi_s(y) = y-s\varphi(y)+\bigo(s^2)$ for $\Phi_s \equiv \Psi_s^{-1}$. It follows that
\begin{align*}
D\Psi_s\big(\Phi_s(y)\big) = I+sD\varphi(y) + \bigo(s^2)
\end{align*}
and
\begin{align*}
 \Big(\det D\Psi_s\big(\Phi_s(y)\big)\Big)^{-1} = 1 -s\,{\rm div}\big(\varphi(y)\big) + \bigo(s^2).
\end{align*}
We may thus express the Dirichlet energy as follows:
\begin{align*}
\E_{\ball{\rho}{y}}\left[u^{\rm in}_s(x)\right] &=\frac12\int_{\ball{\rho}{y}} \Big|\nabla u(y)\,D\Psi_s \big(\Phi_s(y)\big)\Big|^2 \Big(\det D\Psi_s\big(\Phi_s(y)\big)\Big)^{-1} \,\dd y\\
&= \frac12\int_{\ball{\rho}{y}} \Big|\nabla u(y)\,\Big(I+sD\varphi(y) + \bigo(s^2)\Big)\Big|^2 \cdot\Big(1 -s\,{\rm div}\big(\varphi(y)\big) + \bigo(s^2)\Big)\,\dd y.
\end{align*}
Taking $d/ds$ and evaluating at $s=0$, we recover Equation~\eqref{stationary hm} via integration by parts.  

\subsection{Different notions of harmonic maps}

We introduce the following:
\begin{definition}
Let $u \in W^{1,2}(\Omega \subset \R^n; \N \emb \R^p)$.
\begin{enumerate}
\item 
$u$ is a minimising harmonic map on $\Omega$ if it minimises the Dirichlet energy among all the maps $v \in W^{1,2}(\Omega;\N)$ such that $u=v$ on $\p\Omega$ in the sense of trace;

    \item 
$u$ is a weakly harmonic map if it is a weak solution to the Euler--Lagrange equation~\eqref{weak harmonic map eq, 2} for $u$ with respect to outer variations;
\item 
$u$ is a stationary harmonic map if it satisfies~\eqref{weak harmonic map eq, 2} and \eqref{stationary hm, coord}; \emph{i.e.}, if it is a critical point of the Dirichlet energy with respect to both outer and inner variations.
\end{enumerate}
\end{definition}
In fact, we have
\begin{tcolorbox}[
  boxrule=0.5pt,
  colback=white,
  colframe=blue,
  left=1em,
  right=1em,
]
\begin{center} 
$\left\{ \text{weakly harmonic maps} \right\} \subsetneq \left\{ \text{stationary harmonic maps} \right\}  \subsetneq \left\{ \text{minimising harmonic maps} \right\}.$
\end{center}
\end{tcolorbox}
\begin{itemize}
    \item 
    For the first $\subsetneq$, Rivi\`{e}re~\cite{riv} constructed a weakly harmonic map $u_0: \bthree \to \stwo$ whose singular set is the whole domain $\overline{\bthree}$; however, by Naber--Valtorta~\cite{nv}, the singular set of any stationary harmonic map $u: \bthree \to \stwo$ has locally finite 1-dimensional measure.\footnote{It remains an open problem if any stationary harmonic map $u: \bthree \to \stwo$ has finitely many singularities.}

    \item 
    For the second  $\subsetneq$ (see Lin~\cite[p.791]{lin} for the discussions below), Schoen--Uhlenbeck~\cite{su1} proved via $\e$-regularity that any sequence of minimising harmonic maps weakly convergent in $W^{1,2}(\Omega;\N)$ indeed converges strongly in  $W^{1,2}_\loc(\Omega;\N)$. By the Luckhaus lemma~\cite{luck} (see also \cite{hl, su1}), the limiting map is also a minimising harmonic map.

    \noindent
    But the same property fails for minimising harmonic maps: Hardt--Lin--Poon~\cite{hlp} constructed a smooth, stationary, axially symmetric harmonic map $u:\bthree\to\stwo$ with isolated singularities of degree zero. Indeed, $\singu \cap \ball{\delta_0}{0} = \{0\}$ for some $\delta_0>0$ and $\deg \left(u \mres \p\ball{\delta}{0} \right) = 0$ for all $\delta \in ]0,\delta_0]$. However, the rescaled mappings $u_\lambda: x \mapsto u(\lambda x)$ satisfy $u_\lambda \weak {\rm constant}$, while $|\na u_\lambda|^2\dd\leb^3 \mres \bthree \weak 16\pi \hau^1\mres \left[(\{0\}\times \R^1)\cap{\bthree}\right]$ as $\lambda \to 0^+$.

\end{itemize}

\section{Monotonicity formula and $\e$-regularity}\label{sec: mono}

\subsection{Monotonicity identity}
Starting from Equation~\eqref{stationary hm} (or equivalently \eqref{stationary hm, coord}) derived from inner variations, we pick $z\in \Omega$, $0<\rho<r<{\rm dist}(z,\p\Omega)$, and $\varphi \in C^\infty_0\left({\ball{r}{z}};\R^n\right)\cap C^\infty\left(\overline{\ball{\rho}{z}};\R^n\right)$ to obtain that
\begin{align}\label{xx}
&\sum_{i,j=1}^n \int_{\ball{\rho}{z}} \left(\delta_{ij}|\na u|^2-2\p_iu \cdot \p_j u\right)\p_i\varphi^j\,\dd x\nonumber\\
&\qquad = \sum_{i,j=1}^n \int_{\p\ball{\rho}{z}} \left(\delta_{ij}|\na u|^2-2\p_iu \cdot \p_j u\right)\varphi^j\nu^i\,\dd\hau^{n-1},
\end{align}
where $\nu(x) = \frac{x-z}{|x-z|}$ for $x \in \p\ball{\rho}{z}$. Taking $\varphi$ satisfying $\varphi(x)=x-z$ on $\ball{\rho}{z}$ in \eqref{xx}, we get
\begin{align}\label{useful identity}
    (n-2) \int_{\ball{\rho}{z}}|\na u|^2\,\dd x = \rho\int_{\p\ball{\rho}{z}} \left(|\na u|^2 - 2|\na u\cdot\nu|^2\right)\,\dd\hau^{n-1}.
\end{align}
Substituting  \eqref{useful identity} into \eqref{xx}, we obtain
\begin{equation*}
    \frac{d}{d\rho}\left\{\rho^{2-n}\int_{\ball{\rho}{z}}|\na u|^2\,\dd x\right\} = 2\int_{\p\ball{\rho}{z}} \frac{|\na u\cdot \nu|^2}{|x-z|^{n-2}}\,\dd\hau^{n-1}.
\end{equation*}
Hence, for a.e. $0<\sigma<\rho<{\rm dist}(z,\p\Omega)$, we have the monotonicity identity:
\begin{align}\label{monotonicity identity}
  \rho^{2-n}\int_{\ball{\rho}{z}}|\na u|^2\,\dd x - \sigma^{2-n}\int_{\ball{\sigma}{z}}|\na u|^2\,\dd x = 2\int_{\ball{\rho}{z} \setminus \ball{\sigma}{z}} \frac{|\na u(x)\cdot \nu(x)|^2}{|x-z|^{n-2}}\,\dd\hau^{n-1}.
\end{align}
In particular, 
\begin{equation*}
\rho\longmapsto  \rho^{2-n}\int_{\ball{\rho}{z}}|\na u|^2\,\dd x\quad \text{ is nondecreasing for stationary harmonic map $u$}.    
\end{equation*}

\subsection{Density}
By the previous discussions, the following limit exists for every $z \in \Omega$, known as the \emph{density} of $u$ at $z$:
\begin{equation}\label{density, def}
    \Theta_u(z):= \lim_{\rho \to 0^+} \left\{  \rho^{2-n}\int_{\ball{\rho}{z}}|\na u|^2\,\dd x\right\}.
\end{equation}
One may easily show that $\Theta_u(z)$ is upper semicontinuous in $z$, namely
\begin{align*}
    \Theta_u(z) \geq \limsup_{i\to\infty} \Theta_u(z_i)\qquad\text{for } z_i \to z.
\end{align*}
In addition, by monotonicity we have that, for every $z$ such that $\ball{\delta_0}{z} \subset \Omega$,
\begin{align*}
    \Theta_u(z) \leq \delta_0^{2-n} \cdot \E_{\ball{\delta_0}{z}}[u] \leq \delta_0^{2-n}\E_\Omega[u].
\end{align*}

\subsection{Epsilon regularity}
We now prove the (interior) $\e$-regularity theorem, which roughly states that if a harmonic map has small $L^2$-norm and bounded $\dot{W}^{1,2}$-norm, then it is smooth in the interior. The norms are taken in suitable normalised or non-dimensionalised sense.

Throughout, denote $$\dashint_{\Omega'}f\,\dd x := \frac{1}{\leb^n(\Omega')} \int_{\Omega'}f\,\dd x,\qquad (f)_{x_0,r}:=\dashint_{\ball{r}{x_0}}f\,\dd x. $$

\begin{theorem}\label{thm: epsilon reg}
Let $\Lambda'>0$ and $\theta \in ]0,1[$. There exists $\e = \e(n,N,\Lambda',\theta)$ such that the following holds. Let $u \in W^{1,2}(\Omega \subset \R^n;\N)$ be a minimising harmonic map on $\ball{R}{x_0}\subset \Omega$ such that
\begin{align}\label{conditions in e-reg thm}
    R^{2-n}\int_{\ball{R}{x_0}}|\na u|^2\,\dd x \leq \Lambda'\quad\text{and}\quad R^{-n}\int_{\ball{R}{x_0}}\left| u-(u)_{x_0,R} \right|^2\,\dd x < \e^2.
\end{align}
Then $u \in C^\infty(\ball{\theta R}{x_0})$. In addition, for each $j \in \mathbb{N}$ we have the estimate:
\begin{align*}
    R^j\sup_{\ball{\theta R}{x_0}}\left|\na^j u\right| \leq C(n,\N,\Lambda',\theta,j)\,\left\{ R^{-n}\int_{\ball{R}{x_0}}\left| u-(u)_{x_0,R} \right|^2\,\dd x\right\}^{1/2}.
\end{align*}
\end{theorem}

The $\e$-regularity theorem holds also for  \emph{stationary} harmonic maps, under the additional assumption that $\Lambda'$ is also small. That is, the condition~\eqref{conditions in e-reg thm} is replaced with
\begin{align*}
    R^{2-n}\int_{\ball{R}{x_0}}|\na u|^2\,\dd x < \e^2\quad\text{and}\quad R^{-n}\int_{\ball{R}{x_0}}\left| u-(u)_{x_0,R} \right|^2\,\dd x < \e^2.
\end{align*}
It remains unknown whether $\e$-regularity can be proved for stationary harmonic maps under~\eqref{conditions in e-reg thm}. We refer the reader to Bethuel~\cite{b} and Evans~\cite{e}.

\begin{proof}[Sketch of proof of Theorem~\ref{thm: epsilon reg}] Let $u$ be a minimising harmonic map. 
We divide the arguments into five steps.

\smallskip
\noindent
{\bf Step~1. } It suffices to prove for a fixed $\theta$; say $\theta = \frac{1}{8}$. 

Indeed, suppose that the theorem holds for $\theta = 1/8$, and let us consider any other $\theta \in ]0,1[$. Pick $y_1, \ldots, y_Q \in \ball{\theta R}{x_0}$ with $Q=Q(n,\theta)$ such that $\ball{\theta R}{x_0} \subset \bigcup_{j=1}^Q \ball{\frac{(1-\theta)R}{8}}{y_j}$; note that the right-most term lies in $\ball{R}{x_0}$. Since the two quantities in the condition~\eqref{conditions in e-reg thm} are scale-invariant, we deduce from the assertion for $\theta = 1/8$ that
\begin{align*}
&[(1-\theta)R]^j \sup_{\ball{\frac{(1-\theta)R}{8}}{y_k}} \left|\na^j u\right| \\
&\qquad \leq C(n,\N,\Lambda',\theta,j)\,\left\{[(1-\theta)R]^{-n}\int_{\ball{(1-\theta)R}{y_k}}\left| u-(u)_{y_k,(1-\theta)R} \right|^2\,\dd x\right\}^{1/2}.
\end{align*}
Noting that $\lambda \mapsto \int_{\ball{\rho}{z}}|u-\lambda|^2\,\dd x$ attains its minimum at $\lambda = (u)_{z,\rho}$, we take supremum over $k \in \{1,\ldots,Q\}$ on both sides of the above inequality to conclude.

Hence, from now on, let us fix $\theta = \frac{1}{8}$.

\smallskip
\noindent
{\bf Step~2.} We \emph{claim} that the epsilon regularity theorem follows from a PDE argument, which can be schematically summarised as follows:

\begin{tcolorbox}[
  boxrule=0.5pt,
  colback=white,
  colframe=blue,
  left=1em,
  right=1em,
]
\begin{center}
Reverse Poincar\'{e} inequality + Small averaged $L^2$-norm $\Longrightarrow$ Regularity.
\end{center}
\end{tcolorbox}

\begin{lemma}\label{lem: PDE}
Given $\alpha \in ]0,1[$ and $\beta \geq 1$. Then there exists $\e_1 = \e_1(n,\alpha,\beta)>0$ such that the following holds. Let $u=(u^1,\ldots,u^p)^\top \in W^{1,2}\left(\ball{R}{x_0}\subset\R^n;\R^p\right)$ be a weak solution to $$\Delta u = F,$$
where $F \in L^1(\ball{R}{x_0})$ and $|F| \leq \beta |\na u|^2$ a.e. on $\ball{R}{x_0}$. Assume that
\begin{enumerate}
    \item 
    For any $\ball{\rho}{z} \subset \ball{R}{x_0}$, it holds that
    \begin{align}\label{condition a}
        \left(\frac{\rho}{2}\right)^{2-n}\int_{\ball{\frac{\rho}{2}}{y}}|\na u|^2\,\dd x \leq \beta \rho^{-n} \int_{\ball{\rho}{y}}\left|u-(u)_{\rho, y}\right|^2\,\dd x;
    \end{align}

    \item 
    \begin{align}\label{condition b}
    R^{-n} \int_{\ball{R}{x_0}} \left|u-(u)_{x_0,R}\right|^2\,\dd x \leq \e_1^2.
    \end{align}
\end{enumerate}
Then $u \in C^{1,\alpha}\left(\overline{\ball{\frac{R}{4}}{x_0}}\right)$ with 
\begin{align*}
    \left\|u\right\|_{C^{1,\alpha}\left(\overline{\ball{\frac{R}{4}}{x_0}}\right)} \leq C(n,\alpha,\beta) \left\{R^{-n} \int_{\ball{R}{x_0}} \left|u-(u)_{x_0,R}\right|^2\,\dd x\right\}^{1/2}.
\end{align*}
\end{lemma}

We apply the lemma to the harmonic map equation~\eqref{weak harmonic map eq, 2}. Here, $F(u,\na u) = A_u(\na u, \na u)$, verifying the quadratic growth condition $|F|\leq \beta |\na u|^2$ with $\beta = \sup\left\{\left|A_y(\tau,\tau)\right|:\, y \in \N, \, \tau \in T_y^1\N\right\}$, which depends only on the $C^2$-(extrinsic) geometry of $\N$. Then Lemma~\ref{lem: PDE} implies that $u \in C^{1,\alpha}\left(\overline{\ball{\frac{R}{2}}{x_0}}\right)$. But then $\Delta (\p_iu)=\p_iF$ with $|\p_i F| \lesssim |\p A \star \na u \star \na u| + |A \star \na u^{\star 3}| + |A\star \na^2 u \star \na u| \lesssim |\na ^2 u|^2$, where $\star$ denotes generic multilinear combinations. The constants depend on the  $C^{1,\alpha}\left(\overline{\ball{\frac{R}{2}}{x_0}}\right)$-norm of $u$ and the $C^3$-geometry of $\N$. Then $u \in C^{2,\alpha}\left(\overline{\ball{\frac{R}{2}-\frac{R}{8}}{x_0}}\right)$. By bootstrap, we obtain $u \in C^\infty\left(\ball{\frac{R}{4}}{x_0}\right)$ with the estimate
\begin{align*}
    R^j \sup_{\ball{\frac 18 + \frac 1j}{x_0}}\left|\na^j u\right| \leq C(j,n,\N,\alpha)\left\{R^{-n} \int_{\ball{R}{x_0}} \left|u-(u)_{x_0,R}\right|^2\,\dd x\right\}^{1/2}.
\end{align*}
This proves Theorem~\ref{thm: epsilon reg}.

\smallskip
\noindent
{\bf Step~3.} It remains to verify the conditions~\eqref{condition a} and \eqref{condition b} in Lemma~\ref{lem: PDE}. The latter is part of the hypotheses for the $\e$-regularity theorem, so we only need to prove~\eqref{condition a}. Moreover, by the monotonicity of stationary harmonic maps~\eqref{monotonicity identity} and a covering argument as in Step~1, it suffices to prove that for any $\rho_0 \in ]0,R/2[$, $y_0 \in \ball{R/2}{x_0}$, $\rho \leq \frac{\rho_0}{4}$, and $y \in \ball{\frac{\rho_0}{2}}{y_0}$,  
\begin{equation}\label{to prove}
    \rho^{2-n}\int_{\ball{\frac{\rho}{2}}{y}}|\na u|^2\,\dd x \leq C_0(n,\N,\Lambda') \rho^{-n} \int_{\ball{\rho}{y}}\left|u-(u)_{y,\rho}\right|^2\,\dd x, 
\end{equation}
\emph{provided that} 
\begin{align}\label{eps-0 condition}
    \rho_0^{-n} \int_{\ball{\rho_0}{y_0}}\left|u-(u)_{y_0,\rho_0}\right|^2\,\dd x \leq \e_0^2
\end{align}
for suitably small $\e_0=\e_0(n,\N,\Lambda')$.

\smallskip
\noindent
{\bf Step~4.} For \underline{minimising} harmonic maps, this holds by an extension lemma proved in Scheon--Uhlenbeck~\cite{su1}. A nice alternative proof was later given by S. Luckhaus~\cite{luck}. 
\begin{lemma}\label{lem: luckhaus}
    Let $\N \emb \R^p$ be a compact submanifold, $\ball{\rho}{y}$ be an arbitrary ball in the domain of $u$, and  $\Lambda'>0$. Then there exist $\delta_0 = \delta_0(n,\N,\Lambda')>0$ and $C=C(n,\N,\Lambda')>0$ such that the following holds:

    Given $\delta \in ]0,1[$, $u \in W^{1,2}(\ball{\rho}{y};\N)$ satisfying
    \begin{align*}
        \rho^{2-n}\int_{\ball{\rho}{y}}|\na u|^2\,\dd x\leq \Lambda'\quad\text{and}\quad \delta^{-2n}\rho^{-n}\int_{\ball{\rho}{y}}\left|u-(u)_{y,\rho}\right|^2\,\dd x \leq \delta_0^2.
    \end{align*}
There exists a good radius $\sigma \in ]3\rho/4,\rho[$ and a comparison map $w=w_\delta \in W^{1,2}\left(\ball{\rho}{y};\N\right)$ such that $w \equiv u$ near $\p\ball{\sigma}{y}$ and that 
    \begin{align*}
        \sigma^{2-n}\int_{\ball{\sigma}{y}} |\na w|^2\,\dd x \leq \delta\rho^{2-n}\int_{\ball{\rho}{y}}|\na u|^2\,\dd x + \frac{C}{\delta}\rho^{-n}\int_{\ball{\rho}{y}}\left|u-(u)_{y,\rho}\right|^2\,\dd x.
    \end{align*}
\end{lemma}

Let $\delta>0$ be arbitrarily small. Take any  $\rho \leq \frac{\rho_0}{4}$ and $y \in \ball{\frac{\rho_0}{2}}{y_0}$. We first observe that
\begin{equation}\label{C1}
    \rho^{-n}\int_{\ball{\rho}{y}}\left|u-(u)_{y,\rho}\right|^2\,\dd x  \quad\text{ and }\quad \rho^{2-n}\int _{\ball{\rho}{y}}|\na u|^2\,\dd x\leq C_1(n,\N,\Lambda')\delta.
\end{equation}
Indeed, by the monotonicity formula~\eqref{monotonicity identity}, the minimality of $u$, the Luckhaus' Lemma~\ref{lem: luckhaus} (the $\sigma$ below is as in Lemma~\ref{lem: luckhaus}), and the assumptions in Theorem~\ref{thm: epsilon reg}, we obtain that
\begin{align*}
 &\rho_0^{2-n}\int_{\ball{\frac{3\rho_0}{4}}{y_0}} |\na u|^2\,\dd x   \leq C_2(n)\sigma^{2-n}\int_{\ball{\sigma}{z}}|\na u|^2 \,\dd x\\ &\qquad\leq     C_2(n)\sigma^{2-n}\int_{\ball{\sigma}{z}}|\na w|^2\,\dd x\\
    &\qquad\leq C_2(n)\left\{ \delta\rho_0^{2-n}\int_{\ball{\rho_0}{y}}|\na u|^2\,\dd x + \frac{C_3(n,\N,\Lambda')}{\delta}\rho_0^{-n}\int_{\ball{\rho_0}{y}}\left|u-(u)_{y,\rho_0}\right|^2\,\dd x\right\}\\
    &\qquad\leq C_2(n)\delta\Lambda' + C_2(n)C_3(n,\N,\Lambda')\delta^{2n-1} \delta_0^{2n} \\
    &\qquad \leq C_4(n,\N,\Lambda')\delta.
\end{align*} 
Thus, \eqref{C1} follows from the monotonicity formula~\eqref{monotonicity identity} and the Poincar\'{e} inequality.

Then, for any $r>0$ and $z \in \Omega$ such that $\ball{2r}{z} \subset \ball{\rho}{y}$, we claim that
\begin{align}\label{C5}
    r^2\int_{\ball{r/2}{z}}|\na u|^2\,\dd x &\leq C_5(n,\N,\Lambda') \left\{ \delta r^2\int_{\ball{r}{z}}|\na u|^2\,\dd x + \frac{1}{\delta} \int_{\ball{\rho}{y}} \left| u- (u)_{y,\rho}\right|^2\,\dd x \right\}.
\end{align}
The key point is that $C_5$ does not depend on $\delta$.  

To this end, set $\e_0=\e_0(\delta,n,\N,\Lambda')$ to be determined such that the condition~\eqref{eps-0 condition} holds. In particular, we may choose 
\begin{equation}\label{choice of eps}
\e_0 = \e_0(\delta, n, \N,\Lambda') :=\delta_0 \cdot \delta^n    
\end{equation}
for $\delta_0$ as in Luckhaus' Lemma~\ref{lem: luckhaus}; here $\delta>0$ is arbitrary. By a similar argument leading to \eqref{C1} above, we deduce that
\begin{align*}
    &  r^{2-n}\int_{\ball{r/2}{z}}|\na u|^2\,\dd x\\ &\qquad\leq C_6(n)\left\{ \delta r^{2-n}\int_{\ball{r}{z}}|\na u|^2\,\dd x + \frac{C_7(n,\N,\Lambda)}{\delta}r^{-n}\int_{\ball{r}{z}}\left|u-(u)_{y,r}\right|^2\,\dd x\right\}.
\end{align*}
In addition, $$\int_{\ball{r}{z}}\left|u-(u)_{y,r}\right|^2\,\dd x \leq \int_{\ball{r}{z}}\left|u-(u)_{y,\rho}\right|^2\,\dd x \leq \int_{\ball{\rho}{y}}\left|u-(u)_{y,\rho}\right|^2\,\dd x $$ by minimising $\lambda \mapsto \|u-\lambda\|_{L^2(\ball{r}{z})}$. Hence \eqref{C5} follows.

\smallskip
\noindent
{\bf Step~5.} Now we are ready to conclude. By selecting $\delta = \delta(n,\N,\Lambda')$ sufficiently small, we have
\begin{equation}\label{iterate}
 r^{2-n} \Phi(\ball{r/2}{z}) \leq \delta'r^{2-n}\Phi(\ball{r}{z}) + I
\end{equation}
for all $z,r$ such that $\ball{2r}{z}\subset \ball{\rho}{y}$, where $\Phi(\Omega'):=\E_{\Omega'}[u]$ and $$I=C_8(n,\N,\Lambda',\delta'
) \int_{\ball{\rho}{y}}\left|u-(u)_{y,\rho}\right|^2\,\dd x $$ for some $C_8$. Here $\delta'>0$ can be arbitrarily small. A standard iteration in elliptic theory due to De Giorgi yields, by iterating \eqref{iterate} over dyadic annuli, that there exists $\delta'_0 = \delta_0'(n)$ satisfying
\begin{equation*}
    \rho^{2-n}\Phi(\ball{\rho/2}{y}) \leq C_{9}(n)I\qquad\text{whenever } \delta' \in ]0,\delta_0'].
\end{equation*}
 This is precisely \eqref{to prove} that we want to prove.    \end{proof}

\subsection{Regular and singular sets} 

\begin{definition}
Given $u \in W^{1,2}(\Omega;\N)$, where $\Omega \subset \R^n$ is an open set and $\N$ is a Riemannian manifold. Define the regular and singular sets of $u$ respectively by
\begin{equation*}
    \begin{cases}
        {\rm reg}(u) := \Big\{x \in \Omega: \text{there exists a neighbourhood of $x$ on which $u$ is $C^\infty$}\Big\};
        \\
        \singu := \Omega \setminus {\rm reg}(u).
    \end{cases}
\end{equation*}
\end{definition}

We can recast the $\e$-regularity Theorem~\ref{thm: epsilon reg} as follows. Also recall the density from \eqref{density, def}.
\begin{corollary}\label{cor: eps-reg}
Let $u \in W^{1,2}(\Omega\subset\R^n;\N)$ be a stationary harmonic map. There exists $\e=\e(n,\N)$ such that if $\rho^{2-n}\int_{\ball{\rho}{y}} |\na u|^2\,\dd x<\e$ for some $\ball{\rho}{y}\subset \Omega$, then $y \in {\rm reg}(u)$. Moreover, $y \in {\rm reg}(u)$ if and only if $\Theta_u(y)=0$. 
\end{corollary}

This corollary allows us to conclude that the singular set is at most $(n-2)$-dimensional. Indeed, we have the following stronger result:
\begin{theorem}\label{thm: n-2 measure of singular set is zero}
    Let $u\in W^{1,2}(\Omega\subset\R^n;\N)$ be a stationary harmonic map. Then $$\hau^{n-2}(\singu)=0.$$
\end{theorem}

\begin{proof}[Proof of Theorem~\ref{thm: n-2 measure of singular set is zero}]
First of all, note that $\singu$ is a Lebesgue $\leb^n$-null set, for otherwise it cannot have finite Dirichlet energy by Corollary~\ref{cor: eps-reg}. Now fix any compact set $\K\subset\Omega$ such that ${\rm dist}(\K,\p\Omega) \geq \delta_0>0$. If $y\in\singu\cap\K$, then for all $\rho<\delta_0$,
\begin{align*}
    \rho^{2-n}\int_{\ball{\rho}{y}}|\na u|^2\,\dd x \geq \e > 0
\end{align*}
for $\e$ as in the epsilon-regularity Theorem~\ref{thm: epsilon reg}. We next fix some $\delta<\delta_0$ and cover $\K \cap \singu$ by balls $\{\ball{\delta}{y_j}\}_{j=1}^Q$, such that $y_j \in \K\cap\singu$, $\{\ball{\delta/2}{y_j}\}_{j=1}^Q$ are pairwise disjoint, and $Q$ is the maximal integer such that this cover exists. Then
\begin{align*}
    Q\e &\leq \left(\frac{\delta}{2}\right)^{2-n} \int_{\bigsqcup_{j=1}^Q\ball{\delta/2}{y_j}} |\na u|^2\,\dd x.
\end{align*}
Since $\bigsqcup_{j=1}^Q\ball{\delta/2}{y_j} \subset [\K\cap\singu]+\ball{\delta}{0}$, we thus have
\begin{align*}
    Q\delta^{n-2} \leq \frac{2^{n-2}}{\e} \int_{[\K\cap\singu]+\ball{\delta}{0}}|\na u|^2\,\dd x.
\end{align*}
But $\singu$ is an $\leb^n$-null set, sending $\delta \to 0^+$ yields that
\begin{align*}
\lim_{\delta\to 0^+}    Q\delta^{n-2} = 0
\end{align*}
thanks to the dominated convergence theorem, where 
\begin{align*}
    Q\delta^{n-2} = \sum\Big\{\text{radii of the balls in the cover $\{\ball{\delta}{y_j}\}$ for $\K\cap\singu$}\Big\}^{n-2}.
\end{align*}
We may now conclude from the definition of Hausdorff measure.     
\end{proof}

\section{Blow-up analysis and tangent maps}\label{sec: blowup}

\subsection{Definition and basic properties}

To study the geometric-analytic properties of singularities of harmonic maps, we zoom in by rescaling at potential singularities and investigate the ``tangent maps''; \emph{i.e.}, the blow-up limits that exhibit, by construction, various symmetric properties. This is analogous to the study of tangent cones for minimal surfaces.  
 
\begin{definition}\label{def: blowup}
    For $u: \Omega \to \N$ and $\ball{\rho_0}{y}\Subset \Omega$, $\rho \in ]0,\rho_0]$, consider the blow-up of $u$ at $y$:
    \begin{equation*}
\left(T_{y,\rho}u\right)(x) \equiv u_{y,\rho}(x) := u(y+\rho x) \qquad \text{for } x \in \ball{1}{0}.
    \end{equation*}
\end{definition}
We observe that $u_{y, \rho_j}$ converges weakly in $W^{1,2}$ for $\{\rho_j\}\to 0^+$:
\begin{lemma}\label{lem: convergence to tangent map}
For a sequence of radii $\{\rho_j\} \to 0^+$ and a stationary harmonic map $u: \Omega \subset \R^n \to \N$, for each $y \in \Omega$ there exists a subsequence $\{\rho_{j'}\} \subset \{\rho_j\}$ such that $u_{y,\rho_j}$ converges weakly to $\phi: \R^n \to \N$. If $u$ is furthermore minimising, then $\phi$ is also minimising, and the convergence $u_{y,\rho_{j'}
} \to \phi$ is strong in $W^{1,2}_\loc$. 
\end{lemma}

\begin{definition}
    In Lemma~\ref{lem: convergence to tangent map} above, $\phi$ is said to be a tangent map of $u$ at $y$.
\end{definition}

Ding--Li--Li~\cite{dll} constructed examples of a weakly $W^{1,2}$-convergent sequence of stationary harmonic maps whose weak limit is not stationary. Lin~\cite{lin} showed that if $u$ is stationary and $\N$ does not carry harmonic 2-spheres, then $\phi$ is stationary. Also, Li--Tian \cite{lt} showed that if $n=3$ and $u$ is stationary, then $\phi$ is stationary. To the author's knowledge, it remains unknown if this holds for domain manifolds of dimension $\geq 4$; the counterexample in \cite{dll} does not arise from the blow-up of stationary harmonic maps as in Definition~\ref{def: blowup}.

\begin{proof}[Sketch of proof of Lemma~\ref{lem: convergence to tangent map}]

We first note that the $\dot{W}^{1,2}$-norm of $\left\{u_{y,\rho_j}\right\}_{j \in \mathbf{N}}$ is uniformly bounded. Then, by Poincar\'{e}'s inequality, the $W^{1,2}$-norm is uniformly bounded, from which weak convergence follows. Indeed, let $\ball{\rho_0}{y}\subset \Omega$ and $\gamma \in ]0,1[$ such that $\gamma\rho<\rho_0$. Then by a change of variables $x=y+\rho z$ and the monotonicity formula~\eqref{monotonicity identity}, 
\begin{align*}
    (\gamma\rho)^{2-n}\int_{\ball{\gamma\rho}{y}} |\na u(x)|^2\,\dd x &= \gamma^{2-n}\int_{\ball{\gamma}{0}}\left|\na u_{y,\rho} (z)\right|^2\,\dd z \\
    &\leq \rho_0^{2-n}\int_{\ball{\rho_0}{y}} |\na u(x)|^2\,\dd x.
\end{align*}
Thus, for any $\{\rho_j\}\to 0^+$ and fixed $\gamma \in ]0,1[$ we have that
\begin{align*}
    \limsup_{j \to \infty} \int_{\ball{\gamma}{0}}\left|\na u_{y,\rho_j} (z)\right|^2\,\dd z < \infty.
\end{align*}

Now assume that $u$ is minimising. Then for any $\delta>0$, by pigeonholing we may select a radius $\rho$ such that modulo subsequences,
\begin{align}\label{ineq}
    \rho_0^{2-n}\int_{\ball{\rho(1+\e)}{y} \setminus \ball{\rho}{y}} |\na u_{y,\rho_{j'}}|^2\,\dd x \leq \delta.
\end{align}
By weak convergence, the same inequality holds for $\phi$. By a variant of the Luckhaus Lemma~\ref{lem: luckhaus}, one may construct for each $j'$ a comparison map $w_{j'}$, which coincides with $\phi$ inside $\ball{\rho}{y}$ and with $u_{j'}$ outside $\ball{\rho(1+\e)}{y}$, and also satisfies the bound~\eqref{ineq}. We then conclude from the energy minimality of $u_j$, the comparison with $w_{j'}$, as well as the weak convergence $u_{y, \rho_j'} \weak \phi$ in $W^{1,2}$ to conclude that $\phi$ is minimising, from which the strong convergence follows.  \end{proof}

The second paragraph of the previous proof has been streamlined: we omit the details for the choice of the scale $\rho$ and the construction of the comparison map $w_{j'}$. In fact, the same argument leads to the following important
\begin{theorem}[Compactness Theorem]\label{thm: cptness}
    Let $\{u_j\}$ be a sequence of minimising harmonic maps in $W^{1,2}(\Omega;\N)$ such that $\sup_j \|\na u_j\|_{L^2(\ball{\rho}{y})}<\infty$ for all $\ball{\rho}{y} \Subset \Omega$. Then there exists some subsequence $\{u_{j'}\}$ such that $u_{j'} \to u$ strongly in $W^{1,2}\left(\overline{\ball{\rho}{y}};\N\right)$  for all $\ball{\rho}{y} \Subset \Omega$, where $u$ is also energy minimising. 
\end{theorem}

\begin{proof}[Proof of Theorem~\ref{thm: cptness}] 
We refer to Schoen--Uhlenbeck~\cite{su1}, Hardt--Lin~\cite{hl}, and Luckhaus~\cite{luck}.
\end{proof}

\begin{lemma}\label{lem: basic properties of harmonic maps}
Let $u: \Omega \subset \R^n \to \N$ be a stationary harmonic map, and let $\phi: \R^n \to \N$ be a tangent map of $u$ at any point. Then $\phi$ is 0-homogeneous (namely, $\phi(\lambda x)=\phi(x)$ for any $x \in \R^n$ and $\lambda>0$). If, in addition, that $u$ is energy-minimising, then \begin{equation}\label{two Theta}
    \Theta_u(y) = \Theta_\phi(0),
\end{equation}
and hence $y \in {\rm reg}(u)$ if and only if there exists a constant tangent map of $u$ at $y$.
\end{lemma}

\begin{example}\label{example: hlp}
Hardt--Lin--Poon~\cite{hlp} constructed a stationary axially symmetric harmonic map $u: \bthree \to \stwo$ such that $\singu=\{0\}$, $u_{\rho_j,0} \weak q = $ constant weakly in $W^{1,2}_\loc$, but $\frac 12|\na u_{\rho_j,0}|^2\leb^3\mres \bthree$ converges as Radon measures to $8\pi \hau^1\mres L$, where $L$ is the symmetry axis $\{z=0\}\cap \overline{\bthree}$. That is, all the kinetic energy goes to the defect measure in the blow-up limit.

\end{example}

\begin{proof}[Proof of Lemma~\ref{lem: basic properties of harmonic maps}]
Observe by a change of variables that
\begin{align*}
    R^{2-n}\int_{\ball{R}{y}}|\na u_{\rho_j, y}|^2\,\dd x =  (\rho_j R)^{2-n}\int_{\ball{R \rho_j}{y}}|\na u|^2\,\dd x.
\end{align*}
Thus, for any $0<r<R<{\rm dist}(y,\p\Omega)$, we  deduce from the monotonicity identity~\eqref{monotonicity identity} and the existence of density function $\Theta_u(y)$ that
\begin{align*}
    \lim_{j \to \infty} \int_{\ball{R}{y} \setminus \ball{r}{y}} |x|^{2-n} \left|\frac{\p u_{\rho_j,y}}{\p r}\right|^2\,\dd x = 0.
\end{align*}
In particular, $ \frac{\p u_{\rho_j,y}}{\p r} \to 0$ strongly in $L^2(\ball{R}{y} \setminus \ball{r}{y})$. But by Lemma~\ref{lem: convergence to tangent map} we have $u_{\rho_j,y} \weak \phi$ weakly in $W^{1,2}$. As $r$ and $R$ are arbitrary, we thus conclude that $\frac{\p\phi}{\p r}=0$ \textit{a.e.} on $\R^n\setminus \{0\}$.

For the second statement, recall that $y$ is a regular point if and only if the density $\Theta_u(y) = 0$, by $\e$-regularity Theorem~\ref{thm: epsilon reg}. For minimising harmonic maps we have $u_{\rho_{j},y} \to \phi$ strongly in $W^{1,2}$ modulo subsequences (Lemma~\ref{lem: convergence to tangent map}), so $\Theta_{\phi}(0) = \Theta_u(y)$.  \end{proof}

A direct consequence of the 0-homogeneity of tangent map $\phi$ is that it is discontinuous at $0$ unless it is constant. Also observe that:

\begin{lemma}\label{lem: maximisation of density for tgt map}
Let $\phi$ be a tangent map of a stationary harmonic map $u$. The density function $\Theta_\phi$ is  maximised at $0$. Meanwhile, if $\Theta_\phi$ is maximised at $y \in \R^n$, then $\frac{x-y}{|x-y|}\cdot \na\phi(x) = 0$ a.e..
\end{lemma}

In other words, if the density function of $\phi$ attains its maximum at a point, then $\phi$ is radially symmetric with that point as the centre.

\begin{proof}[Proof of Lemma~\ref{lem: maximisation of density for tgt map}]
In view of the monotonicity formula~\eqref{monotonicity identity}, the definition of density, and the 0-homogeneity of the tangent map (Lemma~\ref{lem: basic properties of harmonic maps}), we compute that
\begin{align*}
&    \Theta_\phi(y) + 2 \int_{\ball{\rho}{y}} |x-y|^{2-n} \left|\frac{\p\phi}{\p R_y} \right|^2\,\dd x\\
&\qquad= \rho^{2-n}\int_{\ball{\rho}{y}} |\na \phi|^2\,\dd x \\
    &\qquad \leq \rho^{2-n}\int_{\ball{\rho+|y|}{0}} |\na \phi|^2\,\dd x \\
    &\qquad \leq  \left(1+\frac{|y|}{\rho}\right)^{n-2} (\rho+|y|)^{2-n}\int_{\ball{\rho+|y|}{0}} |\na \phi|^2\,\dd x\\
    &\qquad \equiv  \left(1+\frac{|y|}{\rho}\right)^{n-2} \Theta_\phi(0),
\end{align*}
where $\frac{\p}{\p R_y}(x) = \frac{x-y}{|x-y|}\cdot \na_x$. Send $\rho \to +\infty$ to conclude.   \end{proof}

Tangent maps, even for minimising harmonic maps, may fail to be unique in general. Distinct $\phi$ may be obtained as weak subsequential blow-up limits along different subsequences of $\{\rho_j\}$. White~\cite{w} proved that there exists a 5-dimensional $C^\infty$-target manifold $\N$ and an open set $U$ of smooth boundary maps, such that any minimising harmonic map ${\bf B}^4 \to \N$ with boundary data in $U$ has an isolated singularity $x$ and a continuum of tangent maps at $x$. This won't happen for minimising harmonics maps $\M^3 \to \N^2$; see Gulliver--White~\cite{gw}. On the other hand, if the target manifold $\N$ is real-analytic rather than $C^\infty$, then a seminal result by Simon~\cite{simon'} shows that any minimising harmonic map has a unique tangent map at an isolated singularity. Quantitative information of the convergence rate to the unique tangent map is also obtained~\cite{simon'}.

\subsection{Spine and symmetries of tangent map}

We now connect the study of tangent maps to the analysis of singularities of harmonic maps. Motivated by Lemma~\ref{lem: maximisation of density for tgt map}, let us introduce
\begin{definition}
The spine of a tangent map $\phi$ is defined as
\begin{equation}\label{def: S}
    \bigs (\phi):=\Big\{y \in \R^n:\, \Theta_\phi(y) = \Theta_\phi(0)\Big\}.
\end{equation}
\end{definition}

\begin{proposition}\label{propn: S}
    Let $\phi:\R^n\to\N$ be the tangent map of a stationary harmonic map $u: \Omega \subset \R^n \to \N$ at $y \in \Omega$. Then 
    \begin{itemize}
        \item 
    $\bigs(\phi)$ is a linear subspace of $\R^n$. Also, $\bigs(\phi) = \{y \in \R^n:\, \phi(x+y)=\phi(x)\text{ for all } x \in \R^n\}$. 
    \item 
    If $\phi$ is nonconstant, then $\bigs(\phi) \subset \sing(\phi)$.
    \end{itemize}

Furthermore, if $u$ is also a minimising harmonic map, then $y\in\singu$ if and only if $\dim\bigs(\phi) \leq n-1$ for any tangent map $\phi$ of $u$ at $y$.     
\end{proposition}

Example~\ref{example: hlp} by Hardt--Lin--Poon shows that the ``furthermore'' part is false for stationary harmonic maps. 

\begin{proof}[Proof of Proposition~\ref{propn: S}]
Denote momentarily $\bigs'(\phi) = \{y \in \R^n:\, \phi(x+y)=\phi(x)\text{ for all } x \in \R^n\}$. By a simple change of variables and the definition of $\Theta_\phi$, if  $y \in \bigs'(\phi)$, then $\Theta_\phi(y) = \Theta_\phi(0)$. Thus $\bigs'(\phi) \subset \bigs(\phi)$. On the other hand, in view of Lemma~\ref{lem: maximisation of density for tgt map} and the 0-homogeneity of $\phi$ in Lemma~\ref{lem: basic properties of harmonic maps}, 
\begin{align*}
    \phi(x) = \phi(\lambda x) = \phi(y + (\lambda x-y)) = \phi\left(y + \frac{x}{\lambda} - \frac{y}{\lambda^2}\right) = \phi\left(x + \left(\lambda - \frac{1}{\lambda}\right)y\right)
\end{align*}
for any $x \in \R^n$, $\lambda >0$, and $y \in \bigs(\phi)$. As $\lambda \in ]0,\infty[$ is arbitrary, $\left(\lambda - \lambda^{-1}\right)$ varies through $\R$. This yields $\R\bigs(\phi) \subset \bigs'(\phi)$. Thus, $\bigs(\phi) = \bigs'(\phi)$, and this set is closed under scalar multiplication. But by construction $\bigs'(\phi)$ is closed under addition, so $\bigs(\phi)$ is a linear subspace.

If $\phi$ is nonconstant, then $\phi$ is discontinuous at $0$ (by 0-homogeneous; Lemma~\ref{lem: basic properties of harmonic maps}), and hence $0 \in \sing(\phi)$. By the earlier part, $\phi$ is discontinuous at any point in $\bigs(\phi)$, so $\bigs(\phi) \subset \sing(\phi)$.

Finally, again by the first statement, $\dim\bigs(\phi) = n$ if and only if $\phi$ is constant. By the final part in Lemma~\ref{lem: basic properties of harmonic maps}, this is the case if and only if $y \in {\rm reg}(u)$ if $u$ is minimising.    \end{proof}

\subsection{Dimension estimate for singular set of minimising harmonic maps}
Recall that any stationary harmonic map $u$ satisfies $\hau^{n-2}(\singu) =0$ (Theorem~\ref{thm: n-2 measure of singular set is zero}), so $\singu$ is of codimension $\geq 2$. In fact, a key result for minimising harmonic maps is the following:
\begin{theorem}[Schoen--Uhlenbeck~\cite{su1}]\label{thm: codimenion 3}
Let $u: \Omega \subset \R^n \to \N$ be a \underline{minimising} harmonic map. Then the singular set of $u$ is of codimension 3 or higher, namely
$$\dim\singu \leq n-3.$$   
\end{theorem}

\begin{tcolorbox}[
  boxrule=0.5pt,
  colback=white,
  colframe=blue,
  left=1em,
  right=1em,
]
\begin{center}
One major conjecture in the study of harmonic maps is that Theorem~\ref{thm: codimenion 3} remains valid for stationary harmonic maps.
\end{center}
\end{tcolorbox}

\begin{conjecture}\label{conjecture}
Let $u: \Omega \subset \R^n \to \N$ be a \underline{stationary} harmonic map. Then the singular set of $u$ is of at least codimension 3.  
\end{conjecture}

We also bring to the reader's attention the recent breakthrough by Naber--Valtorta~\cite{nv}. 
\begin{theorem}
    Let $u: \Omega \subset \R^n \to \N$ be a minimising harmonic map. Then  $\singu$ is $(n-3)$-rectifiable with locally uniformly finite $(n-3)$-dimensional Hausdorff measure. 
\end{theorem}
The outline of the proof for this result is the theme of the minicourse in the same Winter School at Westlake University taught by Prof.~D. Valtorta. Prior to \cite{nv}, the most general result had been the following, established by Simon~\cite{simon''}: 
\begin{quote}
If $u:\Omega\subset\R^n\to \N$ is energy-minimising with $\N$ compact and real-analytic, then for each closed ball $\ball{\rho}{y}\Subset\Omega$, $\ball{\rho}{y}\cap\singu$ is the union of a finite pairwise disjoint collection of locally $(n-3)$-rectifiable locally compact sets.  
\end{quote}

The starting point of the proof of Theorem~\ref{thm: codimenion 3} is  the analysis of the spine $\bigs(\phi)$ for a tangent map $\phi$ (see \eqref{def: S} and Proposition~\ref{propn: S}) of minimising harmonic maps. We shall take a different approach from Simon's notes~\cite{simon}, in which the stronger result $\dim \bigs_j \leq j$ for all $j$ by Schoen--Uhlenbeck~\cite{su1} (see Definition~\ref{Sj, def} below) is proved. Instead, we present the celebrated arguments in~\cite{federer}, now known as \emph{Federer's dimension reduction trick}, plus the simple calculus fact that any $0$-homogeneous map on $\R^2$ with finite Dirichlet energy is constant.

To illustrate Federer's ideas, we first introduce:
\begin{definition}\label{def: symmetry}
Let $V^k \subset \R^n$ be a linear subspace. A function $\psi: \R^n \to \N$ is symmetric with respect to $V^k$ if $\psi(x)=\psi(x+y)$ for all $y \in V^k$. We say that $\psi$ is $k$-symmetric if it is degree-$0$-homogeneous and symmetric with respect to some $V^k$. 
\end{definition}
By Proposition~\ref{propn: S}, for a tangent map $\phi: \R^n \to \N$:

\begin{tcolorbox}[
  boxrule=0.5pt,
  colback=white,
  colframe=blue,
  left=1em,
  right=1em,
]
\begin{center}
The maximal subspace with respect to which $\phi$ is symmetric is the spine $\bigs(\phi)$.
\end{center}
\end{tcolorbox}
In other words, $\phi$ is $k$-symmetric and not $(k+1)$-symmetric if and only if $\dim\bigs(\phi)=k$.

\begin{proof}[Proof of Theorem~\ref{thm: codimenion 3}]

We first observe the following iterative construction. Roughly speaking, each time we take the tangent map, we gain an extra symmetry. 

\noindent
{\bf Claim:} Let $\phi: \R^n \to \N$ be a tangent map of some minimising map $u: \Omega \subset \R^n \to \N$. If $V \subset \bigs(\phi)$ is a subspace and $y \notin V$, then $V\oplus \langle y \rangle \subset \bigs(\psi)$ for some tangent map $\psi$ of $\phi$ at $y$.

\begin{proof}[Proof of the claim]
Let $\phi_{y, \rho_j} \to \psi$ strongly in $W^{1,2}$. For any $z \in \R^n$ and $t \in \R$, it holds that
    \begin{align*}
        \phi (y + \rho_j(ty+z)) = \phi\left(y + \frac{\rho_j}{1+t\rho_j}z\right)
    \end{align*}
by the $0$-homogeneity of $\phi$. But $\frac{\rho_j}{1+t\rho_j} = \rho_j + \mathfrak{o}(\rho_j)$, so by selecting an unlabelled common subsequence, we have $\psi(z+ty) = \psi(z)$.     \end{proof}

Now, suppose for contradiction that $\dim\singu > n-3$. Then we could select $x_0 \in \singu$ and a tangent map $\phi_0$ of $u$ at $x_0$. If $\sing(\phi_0) = \bigs(\phi_0)$ we terminate here, and if not, we further $x_1 \in \sing(\phi_0) \setminus \bigs(\phi_0)$ and blow up $\phi_0$ at $x_1$ to obtain a tangent map $\phi_1$. By the \emph{claim} above, we have $\bigs(\phi_0) \oplus \langle x_1 \rangle \subset \bigs(\phi_1) \subset \sing(\phi_1)$.  Again, if $\sing(\phi_1) = \bigs(\phi_0)$ we terminate here, and if not, we take $x_2 \in \sing(\phi_1) \setminus \bigs(\phi_1)$ and blow up $\phi_1$ there...

In this way, we get a chain $(\phi_0, \phi_1, \phi_2, \ldots, \phi_k)$, where $k \leq n-2$ and $\phi_k$ is $(n-2)$-symmetric. In particular, $\phi_k$ depends only on two variables: there exists some coordinate system $(x',x'') \in \R^{n-2} \times \R^2$ such that $\phi_k(x',x'')=v(x'')$ for some $v$. But $v:\R^2 \to \N$ is also energy minimising and $0$-homogeneous, $v(x'')=v(r,\theta) = a(\theta)$ for some function $a$ in one angular variable only. Hence,
\begin{align*}
    \int_{\ball{R}{0}}|\na\phi_k|^2\,\dd x &= \int_0^{2\pi}\int_0^R \frac{|a'(\theta)|^2}{r^2}r\,\dd r\,\dd\theta,
\end{align*}
which is finite only when $\phi_{k}$ is constant. This is impossible since $\sing(\phi_k)$ contains the spine of $\phi_k$, which is constructed by adding singular points.  \end{proof}

The above proof fails completely for stationary harmonic maps, since we do not have strong $W^{1,2}$-convergence. In particular, the kinetic energy may drop each time we take a tangent map. Therefore, a natural programme proposed to resolve Conjecture~\ref{conjecture} for stationary harmonic maps is to keep track of the energy drops. A landmark in this direction is the following ``energy quantisation'' theorem by Lin--Rivi\`{e}re~\cite{lr}: At almost every point of an \((n-2)\)-dimensional concentration set, all such energy drops are accounted for by finitely many $2$-dimensional harmonic-sphere bubbles; no positive energy remains hidden in the neck regions, \textit{i.e.}, the annular regions between successive bubbling scales. See \S\ref{sec: final} for further discussion.

\subsection{Stratification of singular sets}
We introduce the following stratification of $\singu$:
\begin{definition}\label{Sj, def}
Let $u: \Omega \subset \R^n \to \N$ be a harmonic map. For $j \in \{0,1,\ldots, n-1\}$, define
\begin{equation*}
    S_k \equiv S_k(u):=\Big\{y \in \singu:\,\dim\bigs(\phi) \leq k \text{ for $\underline{all}$ tangent maps $\phi$ of $u$ at $y$} \Big\},
\end{equation*}
where $\bigs(\phi)$ is the spine of $\phi$ as in \eqref{def: S}.
\end{definition}

Equivalently, we may express
\begin{equation*}
    S_k =\Big\{y \in \singu: \text{any tangent map $\phi$ of $u$ at $y$ is not $(k+1)$-symmetric} \Big\}.
\end{equation*}

For a minimising harmonic map $u$, it follows from Theorem~\ref{thm: codimenion 3} that \begin{align}\label{chain}
    S_0 \subset S_1 \subset \cdots \subset S_{n-3} = S_{n-2} = S_{n-1} = \singu.
\end{align}
Indeed, this can be proved without the dimension estimate for $\singu$: $S_j \subset S_{j+1}$ by construction, and the ``furthermore'' part of Proposition~\ref{propn: S} shows that $S_{n-1}=\singu$ whenever $u$ is minimising. Suppose for contradiction that $S_{n-3} \neq \singu$. Then, by the definition of $\bigs_j$, there would be a tangent map $\phi$ such that $\dim\bigs(\phi) = n-1$ or $n-2$. However, by Proposition~\ref{propn: S} we know that $\bigs(\phi)$ is a linear subspace of $\R^n$ and $\bigs(\phi) \subset \sing(\phi)$; in addition, by the compactness Theorem~\ref{thm: cptness} we know that $\phi:\R^n\to\N$ is a minimising harmonic map.  Thus $\hau^{n-2}(\sing(\phi))=\infty$, which immediately contradicts Theorem~\ref{thm: n-2 measure of singular set is zero}.

The lowest dimensional part 
\begin{align*}
    S_0 = \Big\{y \in \singu:\, \text{all tangent maps of $u$ at $y$ are degree-$0$-homogeneous but not $1$-symmetric}\Big\}
\end{align*}
has the following characterisation:
\begin{proposition}\label{propn: S0 discrete}
    For each $\alpha >0$, the set $S_0 \cap \{\Theta_u=\alpha\}$ is discrete.
\end{proposition}

\begin{proof}[Proof of Proposition~\ref{propn: S0 discrete}] 
The idea of the proof goes back to De Giorgi~\cite{dg}. See also Almgren~\cite{a}.

Suppose for contradiction that there exists a positive number $\alpha$ and a sequence of points $\{y_j\} \subset S_0 \cap \{\Theta_u=\alpha\}$ such that $y_j \to y$ (with $y_j \neq y$ for each $j$). Consider the blow-up sequence $u_j \equiv u_{y, |y_j-y|}$: by  the compactness Theorem~\ref{thm: cptness}, there exists a subsequence of $\{u_{j}\}$ (not relabelled) that converges strongly in $W^{1,2}$ to a tangent map $\phi: \R^n\to\N$, which is also a minimising harmonic map. Note that $\Theta_\phi(0) = \Theta_u(y)=\alpha$ by Lemma~\ref{lem: basic properties of harmonic maps}.

Now, consider $\xi_j := \frac{y_j-y}{|y_j-y|} \in {\bf S}^{n-1}$. By passing to an unrelabelled subsequence, we have $\xi_j \to \xi \in {\bf S}^{n-1}$. A change of variables yields that $\Theta_{u_j}(\xi_j)=\Theta_{u}(y_j) = \alpha$ for each $j$. Thanks to the upper-semicontinuity of density ($\Theta_{w}(z) \geq \limsup_{j\to\infty}\Theta_{w_j}(z_j)$ if $z_j \to z$ and $w_j \to w$ strongly in $W^{1,2}_\loc$), we thus have $\Theta_\phi(\xi) \geq \alpha=\Theta_\phi(0)$. By Lemma~\ref{lem: maximisation of density for tgt map}, $\Theta_\phi(\xi) = \alpha=\Theta_\phi(0)$ and hence $\xi \in \bigs(\phi)$. However, $y \in S_0$ means that $\dim\bigs(\phi)=0$ for all tangent maps $\phi$ of $u$ at $y$. This is absurd since $|\xi|=1$.   \end{proof}

The following theorem holds for the (non-quantitative) stratification of singular sets:
 
\begin{theorem}[Schoen--Uhlenbeck~\cite{su1}; Almgren~\cite{a}]\label{thm: stratification}
Let $u$ be a minimising harmonic map. Then $\dim S_k(u) \leq k$ for all $k$.      
\end{theorem}

\begin{proof}[Proof of Theorem~\ref{thm: stratification}]

Consider the following

\noindent
{\bf Claim.} Given each $\delta>0$ and $y \in S_k(u)$, there exists $\e=\e(u,y,\delta)>0$ such that at \underline{any} scale $\rho \in ]0,\e]$, one can find a $k$-dimensional linear subspace $L=L(y,\rho) \subset \R^n$ such that 
\begin{equation*}
    \eta_{y,\rho} \bigg(\Big\{ x \in \ball{\rho}{y}:\, \Theta_u(x) \geq \Theta_u(y) -\e \Big\}\bigg) \subset L + \ball{\delta}{0},\tag{$\clubsuit$}
\end{equation*}
where $\eta_{y,\rho}$ is the ``blow-down'' of scale $\rho$ at centre $y$:
\begin{align*}
    \eta_{y,\rho}(x):= \frac{x-y}{\rho}.
\end{align*}
In other words, the blow-down image of the set of points in $\ball{\rho}{y}$ whose density drop is no more than $\e$ is contained in the $\delta$-neighbourhood of some $k$-plane $L$.

Two remarks are in order: 
\begin{itemize}
    \item 
First, even if one fixes $y$, the $k$-plane $L=L(y,\rho)$ may still change as $\rho$ varies. This reflects the fact that tangent maps at a given point may be non-unique when passing to the blow-up limits along different subsequences. 
    \item 
 Second, ($\clubsuit$)  is a one-sided control: it only says that the set on the left-hand side (call it $G$ temporarily) is close to the $k$-plane $L$, but not the other way round. That is, it does not assert that the Hausdorff distance between $G$ and $L$ are small. Note, however, that only the two-sided Hausdorff closeness between $G$ and $L$ on all scales, \textit{i.e.}, the smallness of Jones $\beta$-numbers, would be sufficient to apply the (classical/rectifiable/$W^{1.p}$) Reifenberg theorems; \emph{cf.} \cite{nv}.   
\end{itemize}

Assume the \emph{claim} and fix $\delta>0$. Then we have $$S_k = \bigcup_{j=1}^\infty S_{k,j} = \bigcup_{j=1}^\infty \bigsqcup_{q=1}^\infty S_{k,j,q},$$ where\footnote{For small $j$, it is very likely that $S_{k,j}=\emptyset$.}
\begin{align*}
&    S_{k,j} \equiv S_{k,j}(\rho,\delta) :=  \Big\{y \in S_k:\, \text{($\clubsuit$) holds with $\e = j^{-1}$}\Big\},\\
& S_{k,j,q} \equiv S_{k,j,q}(\rho,\delta) := \Big\{y \in S_{k,j}: \frac{q-1}{j} <\Theta_u(y) \leq \frac{q}{j} \,\Big\}.
\end{align*}
Since the oscillation of density $\Theta_u$ is less than $j^{-1}$ on each $S_{k,j,q}$, by  the \emph{claim}~($\clubsuit$), for each $\rho \leq j^{-1}$ and $y \in S_{k,j,q}$ there exists a $k$-plane $L=L(y,\rho)$ such that 
\begin{align}\label{approx}
\eta_{y,\rho}\left(S_{k,j,q}\cap\ball{\rho}{y}\right) \subset L+ \ball{\delta}{0}.
\end{align} 
Then, by a covering argument (Lemma~\ref{lem: tech}), \eqref{approx} implies that $\hau^{k+\omega(\delta)}(S_{k,j,q})=0$ for some modulus of continuity $\omega: ]0,\infty[ \to ]0,\infty[$. Thus the same holds for $S_k$, and hence $\dim S_k \leq k$. \end{proof}

\begin{proof}[Proof of the claim~$(\clubsuit)$]

We prove by contradiction. Suppose that there were $\delta>0$, $y \in S_k(u)$, and sequences $\{\rho_\ell\}, \{\e_\ell\} \to 0^+$ with $\rho_\ell \leq \e_\ell$ for each $\ell$, such that for every $\ell \in \mathbb{N}$, it holds that 
\begin{align}\label{a, drop}
    &\Big\{ x \in \ball{1}{0}:\, \Theta_{u_{y,\rho_\ell}}(x) \geq \Theta_{u}(y) - \e_\ell \Big\}\nonumber \\ &\text{ is not contained in the $\delta$-neighbourhood of any $k$-plane $L$.}
\end{align}

 On the other hand, $u_{y, \rho_\ell}$ converges strongly in $W^{1,2}$ to a tangent map $\phi$ of $u$ at $y$, with $\Theta_\phi(0)=\Theta_u(y)$; see Lemma~\ref{lem: basic properties of harmonic maps}, \eqref{two Theta}. Here $y \in S_k$, so by Definition~\ref{Sj, def} of $S_k$, the spine $\bigs(\phi)$ (the set on which $\Theta_\phi$ takes its maximum $\Theta_\phi(0)$) has dimension at most $k$. Thus, away from a $k$-plane containing the spine, there must be a positive definite energy drop. More precisely, there exists a $k$-plane $L_\star \supset \bigs(\phi)$ and a positive number $\alpha$ (independent of $\ell$) such that
 \begin{align}\label{b, drop}
     \Theta_\phi(x) < \Theta_\phi(0) - \alpha\qquad\text{whenever $x \in \overline{\ball{1}{0}}$ satisfies ${\rm dist}(x, L_\star)\geq \delta$}.
 \end{align}

However, \eqref{a, drop} and \eqref{b, drop} are inconsistent: By the strong $W^{1,2}$-(sub)convergence of $u_{y, \rho_\ell}\to\phi$ and the upper-semicontinuity of $\Theta$, we deduce from \eqref{a, drop} (taking $L=L_\star$) that for some $x$ with ${\rm dist}(x, L_\star)\geq \delta$, there is no energy drop ($\Theta_\phi(x) \geq \Theta_\phi(0) - \alpha/100$, say). This contradicts~\eqref{b, drop}.  \end{proof}

Finally, let us present in detail  the covering argument in the proof of Theorem~\ref{thm: stratification}. It asserts that given a subset $E$ of $\R^n$, if for any $\delta>0$ one can put pieces of $E$ on all small scales (up to a scale $\rho_0$ that can vary as $\delta$) in the $\delta$-neighbourhood of a $k$-plane $L$, then $E$ is at most $k$-dimensional. Of course, there cannot be any bound on the $k$-dimensional measure of $E$.

\begin{lemma}\label{lem: tech}
Let $E\subset\R^n$ be a  set with the property:
\begin{align*}
    &\forall \delta>0, \, \exists \rho_0=\rho_0(n,\delta) \text{ such that } \forall \rho \in ]0,\rho_0[,\,\forall y\in E,\\
    &\exists \text{$k$-plane } L=L(y,\rho) \subset \R^n \text{ such that } \eta_{y,\rho}\left(E \cap \ball{\rho}{y}\right)\subset L+\ball{\delta}{0}.\tag{$\spadesuit$}
\end{align*}
Then $\hau^{k+\omega(\delta)}(E)=0$ for some modulus of continuity $\omega$.  \end{lemma}

We say that $\omega: ]0,\infty[ \to ]0,\infty[$ is a modulus of continuity if $\omega$ is a non-decreasing function with $\lim_{s \to 0^+}\omega(s) = 0$.

\begin{proof}[Proof of Lemma~\ref{lem: tech}]
Clearly we may assume $\delta \in ]0, 10^{-3}[$ and $E$ is bounded.

Note that for any $k$-plane $L \subset \R^n$, we can cover the $2\delta$-neighbourhood of $L \cap \overline{\ball{1}{0}}$ by balls $\{\ball{4\delta}{y_j}\}_{j=1}^Q$ with centres $y_j \in L \cap {\ball{1}{0}}$ and $Q=Q(k)$ such that $Q(4\delta)^{k+\omega(\delta)} < 1/2$. This can be achieved, \textit{e.g.}, by taking the balls whose concentrically shrunk copies are disjoint, as in the Vitali covering theorem. Thus, by scaling, for any $R>0$ and any $k$-plane $L \subset \R^n$, we can cover the $(2\delta R)$-neighbourhood of $L \cap \overline{\ball{R}{0}}$ by balls $\{\ball{4\delta R}{y_j}\}_{j=1}^Q$ with centres $y_j \in L \cap {\ball{R}{0}}$ and $Q=Q(k)$ such that $Q(4\delta R)^{k+\omega(\delta)} < \frac{1}{2} R^{k+\omega(\delta)}$.

To prove the assertion, it suffices to show that for any given $\e>0$, $E$ can be covered by balls $\{B_i\}_{1}^{P}$ of arbitrarily small radius ${\rm rad}(B_i)=r_i$ such that $P(r_i)^{k+\omega(\delta)} < \e$; here $P$ and $r_i$ may both depend on $\e$. We proceed with an iterative construction: at stage $q \in \mathbb{N}$ such that $2^{-q}T_0<\e$, we take $P=P_q$ and $r_i \equiv \frac{(4\delta)^q\cdot \rho_0}{2}$ for each $i$, such that 
\begin{equation}\label{iterate q}
    P_q\left(\frac{(4\delta)^q\cdot \rho_0}{2}\right)^{k+\omega(\delta)} \leq 2^{-q}T_0,
\end{equation}
where $\rho_0, T_0>0$ are parameters fixed once and for all.

For the above purpose, at the initial state, cover $E$ by $\{\ball{\rho_0/2}{y_j}\}_{j=1}^{P_0}$ and set $T_0 := P_0\cdot \left(\frac{\rho_0}{2}\right)^{k+\omega(\delta)}$. This satisfies \eqref{iterate q} for $q=0$. For each $j \in \{1,\ldots,P_0\}$ and $z_j \in \ball{\rho_0/2}{y_j} \cap E$, by the property~$(\spadesuit)$, we can find a $k$-plane $L_j$ such that $\eta_{z_j,\rho_0/2}(E\cap \ball{\rho_0/2}{z_j}) \subset L_j + \ball{\delta}{0}$. Then, by the beginning part of this proof, we can cover the $(2\delta \rho_0)$-neighbourhood of $L_j \cap \overline{\ball{\rho_0/2}{0}}$ by balls $\{\ball{2\delta \rho_0}{z_j}\}_{j=1}^{Q}$ with centres $z_j \in L_j \cap {\ball{\rho_0}{0}}$ and number $Q$ such that $$Q(2\delta \rho_0)^{k+\omega(\delta)} < \frac{1}{2} \left(\frac{\rho_0}{2}\right)^{k+\omega(\delta)}.$$ Note that $P_1 = P_0Q$ as $Q$ is the same for each $j \in \{1,2,\ldots, P_0\}$. This gives \eqref{iterate q} for $q=1$.\footnote{For notational convenience, let us not relabel $z_j$ in this proof.}

Iterating this process and invoking the property~$(\spadesuit)$, at stage $q$ we can find a $k$-plane $L_j^{(q)}$ such that  $\eta_{z_j,(4\delta)^{q-1}\rho_0/2}(E\cap \ball{(4\delta)^{q-1}\rho_0/2}{z_j}) \subset L^{(q)}_j + \ball{\delta}{0}$. Then, we can cover the $[(4\delta)^{q-1}\rho_0/2]$-neighbourhood of $L^{(q)}_j \cap \overline{\ball{(4\delta)^{q-1}\rho_0/2}{z_j}}$ by balls $\left\{\ball{(4\delta)^q\rho_0/2}{z_j}\right\}_1^{Q'}$, such that 
\begin{equation}\label{ind'}
Q'\cdot \left(4\delta\cdot (4\delta)^{q-1}\cdot\frac{\rho_0}{2}\right)^{k+\omega(\delta)}<\frac{1}{2} \left(\frac{(4\delta)^{q-1}\rho_0}{2} \right)^{k+\omega(\delta)}.    
\end{equation}
Here $Q'$ is the same for each $j \in \{1,2,\ldots, P_{q-1}\}$. But by induction hypothesis,  \begin{equation*}
    P_{q-1}\left(\frac{(4\delta)^{q-1}\cdot \rho_0}{2}\right)^{k+\omega(\delta)} \leq 2^{-(q-1)}T_0,
\end{equation*}
which together with \eqref{ind'} yields that
\begin{align*}
    P_q \cdot \left(\frac{(4\delta)^q\cdot \rho_0}{2}\right)^{k+\omega(\delta)}  = P_{q-1}Q' \cdot \left(\frac{(4\delta)^q\cdot \rho_0}{2}\right)^{k+\omega(\delta)}  \leq 2^{-q}T_0.
\end{align*}
This proves \eqref{iterate q} by induction.   \end{proof}

\section{Some recent developments}\label{sec: final}

In this final section, we point out some significant recent developments in the harmonic map theory, mostly centred around stationary harmonic maps. The discussion here is rather brief and superficial: we only hope to attract the reader's attention to these recent breakthroughs.

\subsection{Energy concentration and bubbling}

Fix a positive constant  $\Lambda>0$. Set
\begin{equation}\label{H-lambda, def}
\hau_\Lambda :=\bigg\{u:\Omega \to \N:\,\text{$u$ is a stationary harmonic map such that } \E_\Omega[u] \leq \Lambda\bigg\}.
\end{equation}
In this subsection, without loss of generality we put $\Omega = \ball{1+\delta_0}{0}\subset\R^{n}$.

We consider the defect measures with respect to $W^{1,2}$-convergence for sequences in $\hau_\Lambda$. More precisely, define
\begin{align}\label{def, M}
\M := &\text{$\bigg\{$weak ${W^{1,2}}$-limits of $\mu_i = |\na u_i|^2\,\dd\leb^n\mres \bone$:$\,\{u_i\}\subset \hau_\Lambda$ for some $\Lambda>0\bigg\}$}. 
\end{align}
Then, given any $\mu \in \M$, there is a weakly convergent sequence $\{u_i\} \subset W^{1,2}(\Omega;\N)$ and  a positive Radon measure $\nu$ on $\bone$ such that 
\begin{align}\label{def, nu}
    \mu_i \equiv |\na u_i|^2\,\dd\leb^n\mres \bone \weak \mu = |\na u|^2 \,\dd\leb^n\mres \bone + \nu,
\end{align}
thanks to Fatou's lemma. We call $\nu$ the \emph{defect measure}. It quantifies the failure of strong $W^{1,2}$-convergence of $\{u_i\}$.

To further characterise $\mu$ or $\nu \in \M$, we introduce the \emph{energy concentration set}: 
\begin{align}\label{def, Sigma}
    \Sigma := \bigcap_{r>0} \bigg\{x \in \bone:\, \liminf_{i\to\infty} r^{2-n}\int_{\ball{r}{x}}|\na u_i(y)|^2\,\dd y \geq \e_0 \text{ for some $\e_0>0$} \bigg\},
\end{align}
and denote 
\begin{align*}
    \pi: \M\longrightarrow\F,\qquad \pi(\mu) = \Sigma.
\end{align*}
The idea is that we may characterise $\Sigma$ (which contains $\singu$) only in terms of $\mu \in \M$, without reference to the weakly convergent sequence $\{u_i\}$.

By a delicate blow-up analysis of the defect measure at the energy concentration set, F.-H. Lin in~\cite{lin} established the following theorem. In particular, the defect measure and singular set for a  stationary harmonic map are $(n-2)$-rectifiable.
\begin{theorem}[Lin~\cite{lin}]\label{thm: structure of Sigma}
Let $\mu \in \M$. Then:
\begin{enumerate}
    \item 
    $\Theta(\mu,x):= \lim_{r \to 0^+} r^{2-n}\mu(\ball{r}{x})$ exists for every $x \in \bone \Subset \Omega$. 
    \item 
    $x \in \Sigma$ iff $\Theta(\mu,x) \geq \e_0>0$.
    \item 
    Let $\nu$ be the positive Radon measure defined in \eqref{def, nu}. We have the decomposition:
    \begin{equation*}
        \pi(\mu) \equiv \Sigma = \singu \cup \spt{\nu}\qquad \text{on }\Omega = \ball{1+\delta_0}{0}\subset\R^n.
    \end{equation*}
    \item 
For $x \in \bone$, it holds that
\begin{equation*}
    \nu(x) = \Theta(x) \hau^{n-2}\mres \Sigma,
\end{equation*}
where
\begin{equation*}
    \e_0 \leq \Theta(x) \leq \left(\frac{\delta_0}{2}\right)^{2-n}\Lambda\qquad\text{for }\hau^{n-2}\mres \Sigma \text{ a.e. } x.
\end{equation*}
    
    \item 
    $\Sigma$ and $\nu$ are both $\hau^{n-2}$-rectifiable. 

    \item 
    $u$ is smooth the harmonic away from the concentration set $\Sigma$, which is a relatively closed subset of $\Omega$ with locally finite $\hau^{n-2}$-measure.
    
\end{enumerate} 
    
\end{theorem}

Moreover, by \cite[Lemma~3.1]{lin}, $\hau^{n-2}(\Sigma)>0$ for some $\mu \in \M$, $\Sigma = \pi(\mu)$ if and only if  the target manifold $\N$ carries a harmonic 2-sphere; \emph{i.e.}, there exists a nonconstant smooth harmonic map $\stwo \to \N$. These are the ``bubbles'' in the Sacks--Uhlenbeck theory~\cite{su}. It then follows from \cite{simon''} by Simon that, if $\N$ carries no harmonic 2-spheres, then $\dim\singu \leq n-4$ for any stationary harmonic map $u:\Omega\to\N$. See \cite[Theorem~D]{lin}.

Lin--Rivi\`{e}re established the following energy quantisation identity for sphere-valued stationary harmonic maps~\cite[Theorem~A]{lr}:
\begin{theorem}[Lin--Rivi\`{e}re~\cite{lr}]\label{thm: Lin-Riviere}
Let $\N$ be the standard sphere ${\bf S}^k$, $k \geq 2$, and $u: \Omega \subset \R^n \to \N$ be a stationary harmonic map with defect measure $\nu$. Then  for $\nu(x) = \Theta(x) \hau^{n-2}\mres \Sigma$ as in Theorem~\ref{thm: structure of Sigma} (4), for \textit{a.e.} $x$, the function $\Theta$ is a finite sum of the Dirichlet energy of harmonic 2-spheres in $\N = {\bf S}^k$. 
\end{theorem}

In a recent breakthrough~\cite{nv'},  Naber--Valtorta extended the energy identity Theorem~\ref{thm: Lin-Riviere} to all compact targets by exploiting new ideas and techniques:
\begin{theorem}\label{thm: NV-energy identity}
Theorem~\ref{thm: Lin-Riviere} remains valid for stationary harmonic maps into any compact $C^2$-target manifold $\N$.
\end{theorem}

\subsection{Quantitative stratification}
In the recent seminal work~\cite{nv}, Naber--Valtorta significantly improved the stratification Theorem~\ref{thm: stratification} by establishing the following:
\begin{theorem}\label{thm: nv}
Let $u:\Omega \subset \R^n \to \N$ be a   stationary harmonic map, where $\N$ is any compact Riemannian manifold. The $k^{\text{th}}$-stratum $S_k(u)$ of its singular set is $k$-rectifiable for each $k$. 
\end{theorem}
Indeed, a stronger result has been obtained in \cite{nv}: For $\hau^k$-\textit{a.e.} $x \in S_k(u)$, there exists a \emph{unique} $k$-plane $V^k \subset T_x\Omega$ such that \emph{every} tangent map at $x$ is $k$-symmetric with respect to $V^k$. That is, although it is still an open question whether the tangent map for a stationary harmonic map is unique, it is now known that the spine of tangent maps at $\hau^k$-\textit{a.e.} $x \in S_k(u)$ is unique.

Moreover, for minimising harmonic maps $u: \ball{4}{p} \subset \Omega \to \N$, it is proved that the top-dimensional stratum $S_{n-3}(u) \subset \singu$ is not only $(n-3)$-rectifiable, but its $\hau^{n-3}$-measure is locally uniformly bounded. Indeed, 
\begin{align*}
\hau^{n-3}\Big(S_{n-3}(u) \cap \ball{1}{p}\Big) \leq C(\Omega,\N,\Lambda)\qquad \text{whenever $\dashint_{\ball{2}{p}} |\na u|^2\,\dd x \leq \Lambda$}.
\end{align*}
Moreover, sharp, effective quantitative bounds can be obtained:
\begin{align*}
{\rm Vol}\Big(\left\{|\na u|>r^{-1}\right\}\cap \ball{1}{p}\Big) \leq C(\Omega,\N,\Lambda)r^3.
\end{align*}

One of the key novel ideas introduced in~\cite{nv} is the quantitative stratification of $\singu$. Instead of considering the sets on which $u$ is exactly $k$-symmetric, they consider a quantitative, relaxed version: say that $\ball{r}{x} \subset \Omega$ is \emph{$(k,\e)$-symmetric} if there exists a $k$-symmetric $\tilde{u}: T_x\Omega \to \N$ such that $\dashint_{\ball{r}{x}}\left|u-\tilde{u}\right|^2\,\dd x < \e$. Then, for a stationary harmonic map $u: \ball{4}{p} \subset \Omega \to \N$ and parameters $\e,r>0$, defined the following quantitative strata:
\begin{align*}
&S^k_{\e,r}(u) := \Big\{x\in \ball{1}{p}:\, \text{there exists no $s \in [r,1[$ such that ${\ball{s}{x}}$ is $(k+1,\e)$-symmetric} \Big\},\\
&S^k_\e(u) := \bigcap_{r>0} S^k_{\e,r}(u).
\end{align*}
It then follows that $S^k(u) = \bigcup_{\e>0}S^k_\e(u)$. Quantitative stratification together with the $L^2$-approximation theorem in \cite[Section~7]{nv} and the $W^{1,p}$-Reifenberg theorem \cite[Section~3]{nv}, among other new techniques developed therein, enable Naber--Valtorta to conclude the above seminal results on stationary and minimising harmonic maps. We refer to Prof.~Valtorta's mini-course in the same winter school for a more detailed account on this topic.

\bigskip
\noindent
{\bf Acknowledgement}. The authors thanks the anonymous referee for careful reading and constructive suggestions.

I wish to express my deepest gratitude and admiration to Bob Hardt, who introduced me to the beautiful subjects of harmonic maps and geometric measure theory. I feel extraordinarily fortunate to have been Bob’s final postdoctoral mentee from 2017 to 2020 at Rice University, Houston. His kindness, profound mathematical insight, and genuine care for those around him have left an enduring impression on me and continue to shape me, both as a mathematician and as a person.

I thank Thierry De Pauw for inviting me to deliver this minicourse. I have learnt a great deal from our many mathematical discussions over the past two years. My sincere thanks also go to Fang-Hua Lin for his insightful discussions and generous guidance.

SL is supported by NSFC Projects 12201399, 12331008, and 12411530065, Young Elite Scientists Sponsorship Program by CAST 2023QNRC001, National Key Research $\&$ Development Programs 2023YFA1010900 and 2024YFA1014900, Shanghai Rising-Star Program 24QA2703600, Qi-Guang Scholarship, and Shanghai Frontier Research Institute for Modern Analysis.

	\bigskip
	\noindent
	{\bf Statement of competing interests}. 
	The author declares that there is no conflict of interest.

	\noindent
	{\bf Statement of data availability}.
	Our manuscript has no associated data.

	\noindent
	{\bf AI Statement}. No AI tools in any form have been used in the writing and preparation of this manuscript.

\end{document}